\documentclass{article}

\usepackage{amsmath,amsthm,amsfonts,mathtools,amssymb}
\newtheorem{theorem}{Theorem}[section]
\newtheorem{assumption}[theorem]{Assumption}
\newtheorem{lemma}[theorem]{Lemma}
\newtheorem{corollary}[theorem]{Corollary}
\newtheorem{proposition}[theorem]{Proposition}
\newtheorem{remark}[theorem]{Remark}

\usepackage[style=numeric,maxbibnames=99]{biblatex}
\usepackage{enumerate}
\usepackage{algorithm}
\usepackage[noend]{algpseudocode}
\usepackage{color}
\usepackage{bbm}
\usepackage{placeins}

\usepackage{hyperref}

\usepackage[version=0.96]{pgf}
\usepackage{tikz}
\usepackage{pgfplots}
\usetikzlibrary{arrows,shapes,automata,backgrounds,petri,positioning}
\usetikzlibrary{decorations.pathmorphing}
\usetikzlibrary{decorations.shapes}
\usetikzlibrary{decorations.text}
\usetikzlibrary{decorations.fractals}
\usetikzlibrary{decorations.footprints}
\usetikzlibrary{shadows}
\usetikzlibrary{calc}
\usetikzlibrary{spy}
\usepackage[textsize=tiny]{todonotes}
\usepackage{floatflt}
\usepackage{comment}
\usepackage{threeparttable}

\newcommand{\EE}{\mathbb{E}}
\newcommand{\NN}{\mathbb{N}}
\newcommand{\N}{\mathbb{N}}
\newcommand{\PP}{\mathbb{P}}
\newcommand{\RR}{\mathbb{R}}

\newcommand{\sphere}{\mathbb S}
\newcommand{\dd}{\mathrm{d}}

\newcommand{\cF}{\mathcal{F}}
\newcommand{\cN}{\mathcal{N}}

\newcommand{\cB}{\mathcal{B}}
\newcommand{\cS}{\mathcal{S}}
\newcommand{\bigO}{\mathcal{O}}

\newcommand{\scp}[2]{\left\langle{#1}, {#2}\right\rangle}
\DeclarePairedDelimiter{\abs}{\vert}{\vert}
\DeclarePairedDelimiter{\norm}{\|}{\|}

\let\dim\relax
\DeclareMathOperator{\dim}{dim}
\DeclareMathOperator{\sign}{sign}
\DeclareMathOperator*{\argmax}{argmax}

\newcommand{\AV}{\mathchoice
    {A\hspace{-2.75pt}\mbox{\rotatebox[]{-1}{\raisebox{0.09mm}{$V$}}}}
    {A\hspace{-2.75pt}\mbox{\rotatebox[]{-1}{\raisebox{0.09mm}{$V$}}}}
    {A\hspace{-2.5pt}\mbox{\rotatebox[]{-1}{\raisebox{0.09mm}{$\scriptstyle V$}}}}
    {A\hspace{-2.45pt}\mbox{\rotatebox[]{-1}{\raisebox{0.09mm}{$\scriptscriptstyle V$}}}}
}

\newcommand{\adj}[1]{ #1}

\title{Computing adjoint mismatch of linear maps}
\author{Jonas Bresch\thanks{Institute of Mathematics, Technical University Berlin, Straße des 17. Juni 136, 10623 Berlin, Germany,
\href{mailto:bresch@math.tu-berlin.de}{bresch@math.tu-berlin.de}. \url{https://www.tu.berlin/imageanalysis}.} \and Dirk A. Lorenz\thanks{Center for Industrial Mathematics, Fachbereich 3, University of Bremen, Postfach 33 04 40, 28334 Bremen, Germany, \href{mailto:d.lorenz@uni-bremen.de}{d.lorenz@uni-bremen.de}.} \and Felix Schneppe\thanks{Center for Industrial Mathematics, Fachbereich 3, University of Bremen, Postfach 33 04 40, 28334 Bremen, Germany, \href{mailto:schneppe@uni-bremen.de}{schneppe@uni-bremen.de}.} \and Maximilian Winkler\thanks{Center for Industrial Mathematics, Fachbereich 3, University of Bremen, Postfach 33 04 40, 28334 Bremen, Germany, \href{mailto:maxwin@uni-bremen.de}{maxwin@uni-bremen.de}.}}

\begin{document}
\maketitle

\begin{abstract}
    This paper considers the problem of detecting adjoint mismatch 
    for two linear maps. 
    To clarify, 
    this means that we aim to calculate the operator norm for the difference of two linear maps,
    where for one we only have a black-box implementation 
    for the evaluation of the map, and for the other we only have a black-box for the evaluation of the adjoint map.        
    We give \adj{a} stochastic \adj{algorithm} for which we prove the almost sure convergence to the operator norm.
    The algorithm is a random search method for a generalization of the Rayleigh quotient and uses optimal step sizes.
    Additionally, 
    a convergence analysis is done for the corresponding singular vector 
    and the respective eigenvalue equation.
\end{abstract}



\section{Introduction}
\label{sec:intro}

In this paper we consider the problem 
of the computation of the operator norm $\norm{A-V}$ for the difference of two linear maps $A$ 
and $V$ between finite dimensional real Hilbert spaces \adj{of dimension $d$ and $m$, respectively,} under restrictive conditions:
First, 
we assume that the map $A$ is only available through a black-box oracle that returns $Av$ for a given input $v$ 
and for $V$
we assume to have only access to a black-box oracle that returns $V^*u$ for given input $u$.
Moreover, we consider that it is not possible to store a large number of vectors (of either the input and the output dimension),
which means that we assume that the method has a storage requirement of $\bigO(\max\{m,d\})$.
Finally, we are interested in a method that is able to approximate the value $\norm{A-V}$ 
to any accuracy and also is able to make further incremental improvements if the output at some stage is not good enough.

Concretely, the question of how to approximate $\norm{A-V}$
under the above conditions is relevant in the field 
of computerized tomography where the forward projection 
and its adjoint have to be discretized \cite{buzug2008ct}.
Usually, the forward projection and its adjoint, 
often called backprojection, are discretized in\-de\-pen\-dent\-ly, 
and both maps are not available as matrices as this would quickly surpass 
the available memory for practically relevant cases.
Moreover, 
it turns out that efficient discretizations of the forward projection 
and the backprojection usually lead to maps which are not adjoint to each other, 
see, e.g., the discussion in \cite{bredies2021convergence}.
Consequently, 
there is an interest in solving reconstruction problems 
in computerized tomography by iterative methods under adjoint mismatch~\cite{lorenz2023chambolle,chouzenoux2023convergence,Savanier2021ProximalGA,Chouzenout2021pgm-adjoint,lorenz2018randomized,zeng2000unmatched}. In \cite{lorenz2023chambolle}, 
it is shown that an erroneous adjoint operator in the Chambolle-Pock method 
may still lead to a convergent method but both the error in the solution 
as well as the conditions for convergence depend on the quantity $\norm{A-V}$.

We are not aware of any method to compute $\norm{A-V}$ 
under our assump\-tions and propose a randomized method for this task in this paper.
In the special case of $V=0$, 
an algorithm with the same requirements has been proposed in \cite{BLSW24}, 
that is, for successive approximation of $\|A\|$ without using the adjoint operator $A^*$.
The corresponding method is in the spirit of stochastic gradient ascent.  
\adj{In 
\cite{BMSS25}, the latter approach was generalized to the maximal Rayleigh quotient
$\scp{v}{Av}/\scp{v}{Bv}$
with respect to a Hermitian positive definite linear map 
$B$.}

A standard method to compute operator norm $\norm{L}$ for linear maps $L$ 
is the power method \cite{mises1929} applied to $L^*L$. 
Its convergence rate depends on the spectral gap,
i.e.~the distance between the two largest eigenvalues of $L^{*}L$.
Other methods for this task rely on methods to compute singular values and use
Krylov subspaces \cite{krylov1921}
or the Lanczos method \cite{lanczos1950}.
However, all these methods require the adjoint $L^{*}$.
Sketching methods \cite{martinsson2002randomized}
for the Schatten $p$-norm for even $p$ or in \cite{magdon2011note} for the spectral norm 
(which is the Schatten $\infty$-norm)
lead to lower bounds with probabilistic guarantees, 
however, they need to store a considerable amount of vectors of the size of $L v$, 
which exceeds the storage requirements in this paper.
In~\cite{kuczynski1992estimating}
an (adjoint-free) alternative to the power iteration 
is proposed for maps for which the input and output dimension are equal.
Also under this assumption,
in \cite{kenney1998} they focus on condition estimation
of the operator norm 
and in \cite{dudley2022} they use Monte Carlo techniques for Schatten-$p$ norms.
Methods to solve linear system and
which avoid using adjoints (or transposes) exist since the 1990s
and are called ``adjoint-free (or transpose-free) methods''~\cite{chan1998transpose}, 
e.g. the transpose-free quasi minimal residual method (TFQMR) from~\cite{freund1993tfqmr}.
Some other methods of this type are considered in \cite{brezinski1998transpose,
chan1998transpose,
sonneveld1989cgs,
lorenz2023randomdescent}.
Moreover,
recent work suggest to learn adjoints~\cite{boulle2024operator} 
but this leads to approximations errors that are hard to quantify.

\subsection{Contribution}
\label{sec:contribution}

The contribution of this work is:
\begin{itemize}
\item We propose \adj{a} stochastic \adj{algorithm} 
producing a sequence that con\-ver\-ges almost surely to the operator norm $\norm{A-V}$
for two linear maps $A,V:\RR^{d}\to\RR^{m}$
only using evaluations of the operators $A$ and $V^*$, and calculation of norms and inner products.
\item Our \adj{method works} for any difference of two linear maps between finite dimensional Hilbert spaces, 
i.e. we do not need any assumptions on the dimension of the spaces or on the structure of the linear maps. 
We restrict ourselves to the Euclidean spaces 
for the input- and output space of the operators,
equipped with the Euclidean norm 
and the standard inner product for simplicity.
\item The storage requirement of the \adj{algorithm} we propose is $\bigO(\max\{m,d\})$ 
which is the minimal amount possible, 
more concretely, we only need to store two vectors of the input di\-men\-sion $d$ 
and two vectors of the output dimension $m$ during the iteration.
\end{itemize}
Additionally, 
our \adj{method} produces an approximation of a pair of right and left singular vectors 
for the largest singular value \adj{of the difference $A-V$}
and can be extended to compute further singular vectors.

\subsection{Outline}
\label{sec:outline}

In Section~\ref{sec:architecture} we present the general architecture of the method.
\adj{In Section~\ref{sec:two_stepsizes},
we introduce and discuss our stochastic algorithm and the almost sure existence is proven in Section~\ref{sec:as_negative_cases}. 
The proof of almost sure convergence is done in Section~\ref{sec:convergence_as_two}}.
Finally,
in Section~\ref{sec:convergence_analysis},
we \adj{show results on the convergence rate}.
\adj{In Section~\ref{sec:numerical-experiments},
we study the performance of our proposed method numerically
and compare it to \cite{BLSW24} when $V = 0$. Finally, we provide adjoint mismatch calculations for different implementations of the  Radon transform.}

\subsection{Notation}
\label{sec:notation}

We use $\norm{x}$ and $\scp{x}{y}$ for the standard Euclidean vector
norm and inner product, respectively, $\norm{A}$ for the induced
operator norm of a map $A$, and $I_{d}$ for the identity map on
$\RR^{d}$. The kernel of an operator $A$ is denoted by $\ker(A)$. With
$\cN(\mu,\Sigma)$ we denote the normal distribution with mean $\mu$ and
covariance $\Sigma$ and with
$\sphere^{d-1} \coloneqq \{x\in\RR^{d}\mid \norm{x}=1\}$ we denote the unit
sphere in $d$-dimensional space. 
The uniform distribution on some \adj{compact} set $C$ 
is denoted by $\mathcal{U}_{C}$.
With $\PP(B)$ we denote the
probability of an event $B$ (and the underlying distribution will be
clear from the context).
Moreover $\sigma_{k}(A)$ will always denote the
$k$-th largest singular value of the linear map $A$ (and sometimes the
argument $A$ will be dropped), and especially we use
$\sigma_{1} = \norm{A}$. For a vector $v \in \RR^d$, we denote by
$\{v\}^\perp \coloneqq \left\{ w \in \RR^d \mid \scp{w}{v}=0 \right\}$ the
hyperplane orthogonal to $v$.
By $A^*$, we denote the adjoint operator of
$A$ defined by $\scp{x}{A^{*}y} = \scp{Ax}{y}$. 
We will occasionally
abbreviate ``almost surely'' with a.s. and mean that the described
event will occur with probability one. 
By $\mathbbm{1}_{B}$ we denote the characteristic function of a
set $B$.

\section{Calculation of $\norm{A-V}$}
\label{sec:adjoint-free-A-V}

We recall that for a linear map $\adj{L}$ from
$\RR^d$ to $\RR^m$ (both equipped with some inner products $\scp{\cdot}{\cdot}$ 
and corresponding norms $\norm{\cdot}$) it holds that 
\label{prop:optimal_pair}
\begin{align*}
    \norm{L}
    = 
    \max_{\norm{v} = 1} \norm{\adj{L} v}
    = 
    \max_{\norm{v} = \norm{u} = 1} \scp{u}{\adj{L} v}
\end{align*}
and the critical points  of the latter optimization problem
are the pairs of nor\-mal\-ized left and right singular vectors 
of $\adj{L}$ to a common singular value.
If the singular value $\sigma$ of $L$ has multiplicity $r\leq d$,
i.e. 
\begin{equation*}
    E_\sigma \coloneqq \{v \in \RR^d \mid L^*L v = \sigma^2 v\}
\end{equation*}
has dimension $r$,
then it holds that 
\begin{equation}
    \label{eqn:E_sigma}
  \tilde E_\sigma 
  \coloneqq 
  \left\{
  \begin{bmatrix} u \\ v \end{bmatrix} \in \RR^{m+d}
  \, \middle|\, 
  \begin{bmatrix} 0 & \adj{L} \\ \adj{L}^* & 0 \end{bmatrix}\begin{bmatrix} u \\v \end{bmatrix} = \sigma \begin{bmatrix} u \\v \end{bmatrix}
  \right\}
\end{equation}
has dimension $r$, too.
Moreover, we have $\tilde E_\sigma |_{\RR^d} = E_\sigma$.
The multiplicities of the nonzero singular values
for $\adj{L}$ and $\adj{L}^*$ are the same
and the only multiplicity
which differs is the one of the zero singular value
when $m \neq d$.
Moreover,
if $\{v_i\}_{i = 1}^r$ is an orthonormal basis 
of $E_\sigma$
then 
\begin{equation*}
    \left\{\tfrac{1}{\sqrt{2}}\left[\begin{matrix}
        \tfrac{\adj{L} v_i}{\norm{\adj{L} v_i}} \\ v_i
    \end{matrix}\right]\right\}_{i = 1}^r
\end{equation*} 
is an orthonormal basis of $\tilde E_\sigma$.

\subsection{Architecture of the Algorithm}
\label{sec:architecture}

The starting point for our algorithm is the formulation 
\begin{align}\label{eq:objective-uv}
    \|A - V\| 
    = 
    \max_{\substack{u \in \RR^{\adj{m}}, u\neq 0,\\ v \in \RR^{\adj{d}}, v \neq 0}} \frac{\langle u, (A - V) v\rangle}{\|u\| \cdot \|v\|}
\end{align}
and our construction of algorithms follows the idea in \cite[Alg.~1]{BLSW24}.
\begin{enumerate}
\item \label{ite:1}
  Initialize uniformly at random $u^{0} \sim \mathcal U_{\sphere^{m-1}}$
  and $v^{0} \sim \mathcal U_{\sphere^{d-1}}$, check if $\scp{u^{0}}{Av\adj{^{0}}} - \scp{V^{*}u^{0}}{v^{0}}$ is positive, if not, flip the sign of $u^{0}$.
\item \label{ite:2}
  Independently sample the search directions $x^{k}, w^{k}$
  that are uniformly distributed on the unit sphere \adj{in the} tangent space 
  of $v^{k}$, and $x^{k}$, respectively.
\item \label{ite:3}
  Then calculate the pair of optimal step sizes $\tau^{k}, \xi^k$ by
    \footnote{The maximizer might be not attained, in general;
      we will prove that this happens with probability zero, 
      see Proposition~\ref{prop:first_order_optimality},
      Proposition~\ref{prop:one_reduced_step_size},
      Proposition~\ref{prop:factos_alomst_surely_not_zero},
      and finally Corollary~\ref{cor:normdistance-2-no-sing-vectors-during-iteration}.}
    \begin{align}
      \label{eq:argmax_two}
      (\tau_k, \xi_k) 
      = \argmax_{\tau, \xi} 
      \frac{\scp{u^k + \tau w^k}{(A - V)(v^k + \xi x^k)}^2}{\norm{u^k + \tau w^k}^2 \cdot \norm{v^k + \xi x^k}^2}
    \end{align}
    and update the iterates $u^{k+1} = \tfrac{u^k + \tau_k w^k}{\norm{u^k + \tau_k w^k}}$
    and $v^{k+1} = \tfrac{v^k + \xi_k x^k}{\norm{v^k + \xi_k x^k}}$. 
\end{enumerate}

\begin{remark}[Non-negativity of objective]\label{rem:nonneg-objective}
  The initialization
  and the ascent of the objective in each iteration ensures that we always have 
  \begin{align*}
    \adj{\scp{u^k}{(A - V) v^k}
    =} \scp{u^{k}}{Av^{k}} - \scp{V^{*}u^{k}}{v^{k}}
    \ge 0
  \end{align*}
  throughout the iteration.
\end{remark}

\begin{remark}[Sampling of $x^k$ and $w^{k}$ {\cite[Rem.~2.3]{BLSW24}}]   \label{rem:distribution-x-w}
  To obtain the samples $x^{k}$ and $w^{k}$ 
  we can  sample $y^{k}\sim \cN(0,I_{d})$ and $z^{k}\sim \cN(0,I_{m})$ and set
  \begin{align*}
    x^{k} 
    & = \frac{y^{k} - \scp{y^{k}}{v^{k}}v^{k}}{\norm{y^{k} - \scp{y^{k}}{v^{k}}v^{k}}}
    \quad\text{and}\quad
    w^{k} 
    = \frac{z^{k} - \scp{z^{k}}{u^{k}}u^{k}}{\norm{z^{k} - \scp{z^{k}}{u^{k}}u^{k}}}
  \end{align*}
\end{remark}

\adj{Now, we analyze} the step size choice in \eqref{eq:argmax_two},
for the two independent directions  $w^{k} \in \sphere^{m-1}$ 
and $x^{k} \in \sphere^{d-1}$.
To simplify notation, 
we define 
the function $q : \RR^2 \to \RR$ measuring the square-root of the objective value in \eqref{eq:argmax_two}
\begin{align}
    \label{eq:q_function_two_stepsizes}
    q_{k}(\tau, \xi) 
    & \coloneqq
    \frac{\scp{ (u^{k} + \tau w^{k})}{(A - V) (v^{k} + \xi x^{k})}}{\norm{u^{k} + \tau w^{k}} \cdot \norm{v^{k} + \xi x^{k}}}.
\end{align}

We introduce the following parameters \adj{in each iteration $k \in \NN$}
\begin{align}
\label{eq:q_k_abcd_k}
\begin{tabular}{l@{\hskip2pt}l @{\hskip1cm} l@{\hskip2pt}l}
    $a_k$ & $\coloneqq \scp{u^k}{(A - V) v^k}$ & $b_k$ & $\coloneqq\scp{w^k}{(A - V) v^k}$ \\
    & $\,= \scp{u^k}{Av^k} - \scp{V^*u^k}{v^k},$ &  & $\,=\scp{w^k}{Av^k} - \scp{V^*w^k}{v^k},$ \\
    \\
    $c_k$ & $\coloneqq \scp{u^k}{(A - V) x^k}$ & $d_k$ & $\coloneqq\scp{w^k}{(A - V) x^k}$ \\
    & $\,= \scp{u^k}{Ax^k} - \scp{V^*u^k}{x^k},$ &  & $\,=\scp{w^k}{Ax^k} - \scp{V^*w^k}{x^k},$ \\
\end{tabular}
\end{align}
which can be evaluated under our restrictive assumptions for the two linear maps.
Hence, 
under the stated normalization and orthogonality assumptions, i.e. $\norm{u}=\norm{v} = \norm{w} = \norm{x}=1$, 
$w \in \{u\}^\perp$ and $x \in \{v\}^\perp$
the function in \eqref{eq:q_function_two_stepsizes} can be written as
\begin{align}
    \label{eq:q_function}
    q_k(\tau, \xi) 
    \coloneqq 
    \frac{a_k + b_k \tau + c_k \xi + d_k \tau\xi}{\sqrt{1 + \tau^2}\sqrt{1 + \xi^2}}.
\end{align}
Notice, 
that by the interest of evaluating the (positive) operator norm,
we first restrict our self to the optimality of \eqref{eq:q_function} (which comes from~\eqref{eq:q_function_two_stepsizes}),
and later use the squared objective \eqref{eq:argmax_two} for further convergence analysis.
Both formulations are equivalent since our initialization guarantees that the objective is always non-negative, see Remark~\ref{rem:nonneg-objective}.
\adj{The proposed method is summarized in Algorithm~\ref{alg:OpNorm3}.}

\begin{algorithm}[t!]
\caption{Matrix- and adjoint free norms for $\|A - V\|$ (two step sizes)}\label{alg:OpNorm3}
\begin{algorithmic}[1]
\State \textbf{Initialize} vectors $v^{0}\in\RR^{d}, u^{0}\in\RR^m$ with $\norm{v^{0}} = \norm{u^{0}}=1$ such that $\scp{u^{0}}{Av^{0}} - \scp{V^{*}u^{0}}{v^{0}}\geq 0$
\For{$k=0,1,2,\dots$}
  \State \textbf{Sample} $y^{k}\sim \cN(0,I_{d})$ and $z^{k}\sim \cN(0,I_{m})$
  \State \textbf{Project}
  \begin{align}
    \label{eq:uv-update-step2}
    x^{k} 
    & = \frac{y^{k} - \scp{y^{k}}{v^{k}}v^{k}}{\norm{y^{k} - \scp{y^{k}}{v^{k}}v^{k}}}
    \quad\text{and}\quad
    w^{k} 
    = \frac{z^{k} - \scp{z^{k}}{u^{k}}u^{k}}{\norm{z^{k} - \scp{z^{k}}{u^{k}}u^{k}}}
\end{align}
 and calculate
\begin{align*}
    a_{k} & = \langle u^{k}, A v^{k}\rangle - \langle V^* u^{k}, v^{k}\rangle, &
    b_k & = \langle w^{k}, A v^{k}\rangle - \langle V^* w^{k}, v^{k}\rangle, \\
    c_k & = \langle u^{k}, A x^{k}\rangle - \langle V^* u^{k}, x^{k}\rangle,&
    d_k & = \langle w^{k}, A x^{k}\rangle - \langle V^* w^{k}, x^{k}\rangle.
\end{align*}
\State \textbf{Calculate} step sizes $\tau_{k}$ and $\sigma_{k}$ as 
\Comment{Prop.~\ref{prop:first_order_optimality} and Prop.~\ref{prop:one_reduced_step_size}}
\begin{align*}
    (\tau_{k}, \xi_k)
    & = 
    \argmax_{\tau, \xi} \frac{\langle u^{k} 
    + \tau w^{k}, (A - V)(v^{k} + \xi x^{k})\rangle^2}{\norm{u^{k} + \tau w^{k}}^2 
    \cdot \norm{v^{k} + \xi x^{k}}^2}
\end{align*}
with $e_k \coloneqq a_kb_k + c_kd_k$ and $f_k \coloneqq a_k^2 + c_k^2 - b_k^2 -d_k^2$ as
\begin{align*}
    \tau_k \coloneqq -\text{sign}(e_k)\left(\frac{f_k}{2|e_k|} - \sqrt{\frac{f_k^2}{4e_k^2} + 1}\right),
    \quad
    \xi_k \coloneqq \frac{c_k + \tau_k d_k}{a_k + \tau_k b_k}.
\end{align*}
\State \textbf{Update}
\begin{align}
    \label{eq:v-update-step2}
    v^{k+1} 
    & = \tfrac{v^{k}+ \xi_k x^{k}}{\sqrt{1+\xi_{k}^{2}}} \qquad
    \text{and}\quad
    &
    u^{k+1} 
    & = \tfrac{u^{k}+ \tau_k w^{k}}{\sqrt{1+\tau_{k}^{2}}}
\end{align}
\EndFor
\State \textbf{Return} estimate $\norm{A - V} \approx \scp{u^{k}}{Av^{k}} - \scp{V^*u^k}{v^k}$
\end{algorithmic}
\end{algorithm}


\subsection{On the optimal step sizes}
\label{sec:two_stepsizes}

We give a first result concerning the optimality criteria 
for the optimal pair of step sizes for the function $q_k(\tau,\xi)$ \adj{from \eqref{eq:q_function}},
say $(\tau^*_k, \xi^*_k)$,
in the $k$-th iteration. 

\begin{proposition}
    \label{prop:first_order_optimality}
    The first order optimality criteria for $(\tau, \xi)$ in \eqref{eq:argmax_two}
    within the $k$-th iteration
    are given by 
    \begin{align*}
        0 = b_k + \xi_k^* d_k - \tau_k^*(a_k + \xi_k^* c_k),
        \quad
        0 = c_k + \tau_k^* d_k - \xi_k^*(a_k + \tau_k^* b_k).
    \end{align*}
    Furthermore, 
    if we assume that $a_k\neq 0$, $a_k b_k + c_k d_k\neq 0$, $a_k c_k + b_k d_k \neq 0$ and $a_k d_k-b_k c_k \neq 0$, 
    the first order optimality criteria for \eqref{eq:argmax_two}
    are fulfilled exactly for
    \begin{align*}
        \tau_{k,\pm}^* 
        = 
        - \frac{a_k^2 - b_k^2 + c_k^2 - d_k^2}{2(a_kb_k + c_kd_k)}
        \pm \sqrt{\frac{(a_k^2 - b_k^2 + c_k^2 - d_k^2)^2}{4(a_kb_k + c_kd_k)^2} + 1}
    \end{align*}
    and 
    \begin{align*}
        \xi_{k,\pm}^*
        = 
        - \frac{a_k^2 + b_k^2 - c_k^2 - d_k^2}{2(a_kc_k + b_kd_k)}
        \pm \sqrt{\frac{(a_k^2 + b_k^2 - c_k^2 - d_k^2)^2}{4(a_kc_k + b_kd_k)^2} + 1}.
    \end{align*}
\end{proposition}
\begin{proof}
  We analyze the critical points $(\tau,\xi)$ of $q_{k}$ since they are are also critical points of $q_{k}^{2}$:
    \begin{align*}
        \partial_\tau q_k (\tau, \xi)
        & = 
        \tfrac{(b_k + \xi d_k)(1 + \tau^2) - \tau(a_k + \tau b_k + \xi c_k + \tau\xi d_k)}{\sqrt{1 + \tau^2}^3\sqrt{1 + \xi^2}}
        = 
        \tfrac{b_k + \xi d_k - \tau(a_k + \xi c_k)}{\sqrt{1 + \tau^2}^3\sqrt{1 + \xi^2}}
    \end{align*}
    Hence, we have
    \begin{align}
        0 = \partial_\tau q_k (\tau, \xi)
        \iff
       0 & = b_k + \xi d_k - \tau(a_k + \xi c_k).
       \label{eq:opt_tau}
    \end{align}
    Similarly it holds that
    \begin{align*}
        \partial_\xi q_k (\tau, \xi)
        & = \tfrac{(c_k + \tau d_k)(1 + \xi^2) - \xi(a_k + \tau b_k + \xi c_k + \tau\xi d_k)}{\sqrt{1 + \tau^2}\sqrt{1 + \xi^2}^3}
        = \tfrac{c_k + \tau d_k - \xi(a_k + \tau b_k)}{\sqrt{1 + \tau^2}\sqrt{1 + \xi^2}^3}.
    \end{align*}
    and thus,
    \begin{align}
    	0 = \partial_\xi q_k (\tau, \xi)
    	\iff 
    	0 & = c_k + \tau d_k - \xi(a_k + \tau b_k).
            \label{eq:opt_xi}
    \end{align}
    Hence,
    if $a_k \neq - \xi c_k$ and $a_k \neq - \tau b_k$,
    which we discuss later,
    we have from \eqref{eq:opt_tau} and \eqref{eq:opt_xi} that
    \begin{align}
        \label{eqn:opt_crit_two_stepsizes_solved}
        \tau = \frac{b_k + \xi d_k}{a_k + \xi c_k} 
        \quad\text{and}\quad
        \xi = \frac{c_k + \tau d_k}{a_k + \tau b_k}.
    \end{align}
    Plugging the optimal form of $\tau$ into the optimal form of $\xi$ yields 
    \begin{align*}
        \xi
        & = 
        \frac{c_k + \frac{b_k + \xi d_k}{a_k + \xi c_k} d_k}{a_k + \frac{b_k + \xi d_k}{a_k + \xi c_k} b_k}
        \quad \iff \quad 
        0 = 
        \xi^2 + \xi \frac{a_k^2 + b_k^2 - c_k^2 - d_k^2}{a_k c_k +  d_k b_k} - 1.
    \end{align*}
    We obtain the two optimal solutions
    \begin{align*}
        \xi_{\pm}
        \coloneqq
        -\frac{a_k^2 + b_k^2 - c_k^2 - d_k^2}{2(a_k c_k + d_k b_k)}
        \pm 
        \sqrt{\frac{(a_k^2 + b_k^2 - c_k^2 - d_k^2)^2}{4(a_k c_k + d_k b_k)^2} + 1}.
    \end{align*}
    Similarly,
    by plugging the optimal form of $\xi$ into the optimal form of $\tau$ yields 
    \begin{align*}
        \tau 
        & = 
        \frac{b_k + \frac{c_k + \tau d_k}{a_k + \tau b_k} d_k}{a_k + \frac{c_k + \tau d_k}{a_k + \tau b_k} c_k} 
        \quad \iff \quad
        0 = 
        \tau^2 
        + \tau \frac{a_k^2 - b_k^2 + c_k^2 - d_k^2}{a_k b_k + d_k c_k}
        - 1
    \end{align*}
    and the assertion follows.
    Now,
    we discuss the case,
    where $a_k + \tau b_k = 0$.
    There are two cases:
    1. $\tau = 0$, which results into $a_k = 0$, 
    contradicting the assumptions.
    2. $b_k = 0$, which also implies $a_k = 0$, 
    contradicting the assumptions.
    From the latter, 
    we have $\tau = - \tfrac{a_k}{b_k} \neq 0$,
    otherwise we have again the first case.
    Since $a_k + \tau b_k = 0$,
    we have from the optimality criteria \eqref{eq:opt_xi} that $c_k + \tau d_k = 0$ 
    and plugging in the expression for $\tau = - \tfrac{a_k}{b_k}$ yields 
    \begin{align*}
        0 = c_k - \tfrac{a_k}{b_k} d_k 
        \quad\Rightarrow\quad
        0 = b_k c_k - a_k d_k,
    \end{align*}
    contradicting the assumptions again.
    The same discussion can be done for $a_k + \xi c_k = 0$
    resulting into either $a_k = 0$ or $b_k c_k - a_k d_k = 0$, again.
    In summary,
    we have finite (potential) step sizes $(\tau_k^*, \xi_k^*)$.
\end{proof}

It remains to decide which pair of critical points corresponds to a local maximum of $q_{k}$ \adj{from \eqref{eq:q_function}}.
The previous proposition shows how the optimal $\tau_{k}^{*}$ depends on the respective optimal $\xi_{k}^{*}$.
Hence, 
we can use the optimality criteria for the critical step sizes $(\tau_k^*, \xi_k^*)$ 
to rewrite $q_k(\tau,\xi)$
as a one-dimensional function
in $\tau$ \adj{on the curve $(\tau, \xi_k^*(\tau))$ defined in \eqref{eqn:opt_crit_two_stepsizes_solved}},
which we again denote by 
\begin{equation}
    \label{eq:q_function_red}
    q_k(\tau) := q_{k}(\tau,\xi_{k}^{*}(\tau))
\end{equation} 
(this should not lead to confusion in the following).

\begin{lemma}
    \label{lem:two_step_sizes}
    Under the assumption from Proposition~\ref{prop:first_order_optimality}, 
    $q_k$ \adj{from \eqref{eq:q_function}}can be explicitly parametrized as one-dimensional function in $\tau$
    on the curve $\xi = \xi_k^*(\tau)$
    \begin{align*}
        q_{k}(\tau) 
        = &\;
        \textup{sign}(a_k + \tau b_k)
        \frac{\sqrt{a_k^2 + c_k^2 + 2\tau (a_k b_k + c_kd_k) + \tau^2(b_k^2 + d_k^2)}}{\sqrt{1 + \tau^2}}.
    \end{align*}
\end{lemma}
\begin{proof}
    From the optimality criteria \eqref{eq:opt_xi} for $\xi$,
    c.f.~Proposition~\ref{prop:first_order_optimality}, 
    we observe that
    \begin{align*}
        1 + \xi^2 
        = \frac{a_k^2 + c_k^2 + 2\tau(a_kb_k + c_kd_k) + \tau^2(b_k^2 + d_k^2)}{(a_k + \tau b_k)^2}.
    \end{align*}
    Hence,
    we have 
    \begin{align}
        \label{eq:one_plus_xi^2}
        \sqrt{1 + \xi^2}
        = 
        \frac{1}{|a_k + \tau b_k|}
        \sqrt{a_k^2 + c_k^2 + 2\tau(a_kb_k + c_kd_k) + \tau^2(b_k^2 + d_k^2)}.
    \end{align}
    Using \eqref{eqn:opt_crit_two_stepsizes_solved} \adj{in \eqref{eq:q_function}}, 
    we obtain
    \begin{align}
        q_k(\tau, \xi)
        = \tfrac{a_k + \tau b_k + \xi(c_k + \tau d_k)}{\sqrt{1 + \tau^2}\sqrt{1 + \xi^2}}
        = \tfrac{(1+\xi^2)(a_k + \tau b_k)}{\sqrt{1 + \tau^2}\sqrt{1 + \xi^2}}
        = \tfrac{\sqrt{1 + \xi^2}}{\sqrt{1 + \tau^2}}(a_k + \tau b_k).\label{eq:q_function_tau}
    \end{align}
    Inserting the expression for $\sqrt{1 + \xi^2}$ from \eqref{eq:one_plus_xi^2} yields the assertion.
\end{proof}

\begin{remark}[Non-differentiability]
    \label{rem:diff_q_2}
    Notice that $q_k(\tau)$ in the form of the previous lemma
    is not differentiable due to the sign.
    However,
    \begin{align}
        \label{eq:q_2_function_tau}
        q_k^2(\tau)
        =
        \frac{a_k^2 + c_k^2 + 2\tau (a_kb_k + c_kd_k) + \tau^2(b_k^2 + d_k^2)}{1 + \tau^2}
    \end{align}
    is differentiable again. 
    Moreover,
    $q_k^2(\tau)$ equals $q_k^2(\tau, \xi)$
    on the curve given by \eqref{eq:opt_xi}%
    ---the optimality condition of $\xi$ w.r.t. $\tau$ for \eqref{eq:q_function}
    and therefore \eqref{eq:argmax_two}.
    Hence we will maximize $q_k^2(\tau)$ in the following.
\end{remark}

\begin{proposition}
    \label{prop:one_reduced_step_size}
    \adj{For $q_k$ defined in \eqref{eq:q_function_red}
    and $a_k, b_k, c_k$ and $d_k$ defined as in \eqref{eq:q_k_abcd_k} it holds 
    that:}
    \begin{enumerate}[i)] 
        \item If $a_k b_k + c_k d_k \neq 0$, 
        then the optimal pair $(\tau_{k},\xi_{k})$ solving \eqref{eq:argmax_two} is given by
        \begin{align*}
            \tau_k
            & = 
            \textup{sign}(a_k b_k + c_k d_k)\Bigl(-\tfrac{a_k^2 - b_k^2 + c_k^2 - d_k^2}{2\abs{a_k b_k + c_k d_k}} + \sqrt{\tfrac{(a_k^2 - b_k^2 + c_k^2 - d_k^2)^2}{4(a_k b_k + c_k d_k)^2} + 1}\Bigr)\\
        \xi_{k} & = \tfrac{c_k + \tau_kd_k}{a_k+\tau_kb_{k}}.
        \end{align*}
        \item If $a_k b_k + c_k d_k = 0$ and
        \begin{enumerate}[a)]
            \item $a_k^2 + c_k^2 > b_k^2 + c_k^2$,
            then $\tau = 0$ maximizes $q_k^2(\tau)$,
            \item $a_k^2 + c_k^2 < b_k^2 + c_k^2$,
            then $q_k^2(\tau)$ does not attain its supremum,
            \item $a_k^2 + c_k^2 = b_k^2 + c_k^2$,
            then $q_k^2$ is a constant.
        \end{enumerate}
    \end{enumerate}
\end{proposition}
\begin{proof}
    We have from Remark~\ref{rem:diff_q_2}
    \begin{align}
        \label{eq:root_tau}
        \big(q_k^2\big)'(\tau) 
        & = 
        \tfrac{[2(a_kb_k + c_kd_k) + 2\tau(b_k^2 + d_k^2)](1 + \tau^2)
        - 2\tau [a_k^2 + c_k^2 + 2 \tau(a_k b_k + c_k d_k) + \tau^2(b_k^2 + d_k^2)]}{(1 + \tau^2)^2} \notag \\
        & = 
        2\tfrac{a_kb_k + c_kd_k 
        - \tau (a_k^2 - b_k^2 + c_k^2 - d_k^2)
        - \tau^2(a_kb_k + c_kd_k)}{(1 + \tau^2)^2}.
    \end{align}
    As long as $a_kb_k + c_k d_k \neq 0$
    we have two solutions of $\big(q_k^2\big)'(\tau) =0$
    given by 
    \begin{align*}
        \tau_{\pm}
        \coloneqq
        -\tfrac{a_k^2 - b_k^2 + c_k^2 - d_k^2}{2(a_kb_k + c_kd_k)}
        \pm \sqrt{\tfrac{(a_k^2 - b_k^2 + c_k^2 - d_k^2)^2}{4(a_kb_k + c_kd_k)^2} + 1},
    \end{align*}
    which we already know from Proposition~\ref{prop:first_order_optimality}.
    However,
    we have to decide which one is optimal,
    i.e. solves \eqref{eq:argmax_two}.
    To that end we calculate
    \begin{align*}
        \big(q_k^2\big)'(0) = 2(a_k b_k + c_k d_k).
    \end{align*}
    Hence, we choose the positive root $\tau_+$ if $a_k b_k + c_k d_k > 0$, 
    and the negative root if $a_k b_k + c_k d_k < 0$.
    This shows the claimed formula for $\tau_{k}$ and the formula for the optimal $\xi_{k}$ follows from~\eqref{eq:opt_xi}.
    
    If $a_k b_k + c_k d_k = 0$,
    then $\tau = 0$ is optimal 
    if any only if $a_k^2 + c_k^2 > b_k^2 + d_k^2$.
    If $a_k b_k + c_k d_k = 0$ and $a_k^2 + c_k^2 = b_k^2 + d_k^2=0$,
    then any $\tau$ is optimal as can be seen from \eqref{eq:root_tau}. Finally, if $a_k^2 + c_k^2 < b_k^2 + d_k^2$, then from \eqref{eq:root_tau} we get $\big(q_k^2\big)'(\tau)>0$ for $\tau>0$ and $\big(q_k^2\big)'(\tau)<0$ for $\tau<0$, which means that the maximum is not attained.
\end{proof}

Figure~\ref{fig:shape-q} illustrates the function $q_{k}^{2}(\tau)$ in the different cases (and the situation is similar to the one in \cite[Figure~1]{BLSW24}).

\pgfplotsset{width=5cm,height=4cm,
  xlabel style={xlabel=$\tau$},
  samples = 50,
  every axis plot post/.append style={purple,mark=none},
  every axis post/.append style={xmin=-5, xmax=5, ymin=0, ymax = 1.9, axis lines=middle, xtick=\empty},
  every tick label/.append style={font=\tiny}
}

\begin{figure}[htb]
  \resizebox{\linewidth}{!}{
  \begin{tabular}{cccc}
    & 
    $a_{k}b_k + c_kd_k>0$ &
    $a_{k}b_k + c_kd_k<0$ &
    $a_{k}b_k + c_kd_k=0$ \\\bigskip
    \rotatebox{90}{\scalebox{0.9}{\quad $a_k^2 + c_k^2 < b_k^2 + d_k^2$}} & 
    \begin{tikzpicture}[scale=1.3]
      \begin{axis}[ytick={1,1.5},yticklabels={$a_k^2 + c_k^2$,$b_k^2 + d_k^2$}]
        \addplot{(1+x*1+1.5*x*x)/(1+x*x)};
        \addplot[thin,dotted] {1.5};
      \end{axis}
    \end{tikzpicture}&
    \begin{tikzpicture}[scale=1.3]
      \begin{axis}[ytick={1,1.5},yticklabels={$a_k^2 + c_k^2$,$b_k^2 + d_k^2$}]
        \addplot{(1-x*1+1.5*x*x)/(1+x*x)};
        \addplot[thin,dotted] {1.5};
      \end{axis}
    \end{tikzpicture}&
    \begin{tikzpicture}[scale=1.3]
      \begin{axis}[ytick={1,1.5},yticklabels={$a_k^2 + c_k^2$,$b_k^2 + d_k^2$}]
        \addplot{(1+1.5*x*x)/(1+x*x)};
        \addplot[thin,dotted] {1.5};
      \end{axis}
    \end{tikzpicture}\\\bigskip
    \rotatebox{90}{\scalebox{0.9}{\quad $a_k^2 + c_k^2 > b_k^2 + d_k^2$}} & 
    \begin{tikzpicture}[scale=1.3]
      \begin{axis}[ytick={1,1.5},yticklabels={$b_k^2 + d_k^2$,$a_k^2 + c_k^2$}]
        \addplot{(1.5+x*1+1*x*x)/(1+x*x)};
        \addplot[thin,dotted] {1};
      \end{axis}
    \end{tikzpicture}&
    \begin{tikzpicture}[scale=1.3]
      \begin{axis}[ytick={1,1.5},yticklabels={$b_k^2 + d_k^2$,$a_k^2 + c_k^2$}]
        \addplot{(1.5-x*1+1*x*x)/(1+x*x)};
        \addplot[thin,dotted] {1};
      \end{axis}
    \end{tikzpicture}&
    \begin{tikzpicture}[scale=1.3]
      \begin{axis}[ytick={1,1.5},yticklabels={$b_k^2 + d_k^2$,$a_k^2 + c_k^2$}]
        \addplot{(1.5+1*x*x)/(1+x*x)};
        \addplot[thin,dotted] {1};
      \end{axis}
    \end{tikzpicture}\\
  \end{tabular}}
  \caption{Shape of $q_k^2(\tau)$ from~\eqref{eq:q_2_function_tau} 
  for different cases of the signs of $a_kb_k + c_kd_k$ 
  and $a_k^2 - b_k^2 + c_k^2 - d_k^2$.}
  \label{fig:shape-q}
\end{figure}

\begin{remark}[Rewritten squared objective]
    \label{rem:rewritten_squared_objective}
    Similar to the proof of Lemma \ref{lem:two_step_sizes},
    we could also write 
    for the optimal pair $(\tau_k, \xi_k)$
    \begin{align}
        \label{eq:q_function_xi}
        q_k(\tau_k, \xi_k)
        = 
        \tfrac{\sqrt{1 + \tau_k^2}}{\sqrt{1 + \xi_k^2}} (a_k + \xi_k c_k)
    \end{align}
    by replacing the optimality criteria for $\xi$.
    In summary, 
    combining \eqref{eq:q_function_tau} 
    and \eqref{eq:q_function_xi} we have 
    \begin{align*}
        q_k^2(\tau_k, \xi_k)
        = 
        (a_k + \tau_k b_k)(a_k + \xi_k c_k).
    \end{align*}
    along the two curves given by the optimal criteria 
    \eqref{eqn:opt_crit_two_stepsizes_solved}.
\end{remark}

\adj{\begin{remark}[Markov chain provided by $(u^k, v^k)$]
    \label{rem:distribution-x-w-const}
    The pair $(u^k, v^k)$ 
    generated by Algorithm~\ref{alg:OpNorm3}
    can be understood as sequence of random variables,
    i.e. measurable functions 
    $(u,v)^k : \Omega \to \sphere^{m-1} \times \sphere^{d-1}$,
    realized by $(u^k(w), v^k(w)) \coloneqq (u,v)^k(w)$ for some $w \in \Omega$,
    where $\Omega$ is a sampling space
    equipped with a $\sigma$-algebra $\cF$ and probabilty meausre $\PP$
    forming a probability space $(\Omega, \cF, \PP)$.
    Similar to \cite[Rem.~2.7]{BLSW24},
    the sampling space is $\Omega = (\sphere^{m-1}\times\sphere^{d-1})^{\NN}$,
    and $\cF = \cS^{\otimes \NN}$ with $\cS = \cB(\sphere^{m-1}) \otimes \cB(\sphere^{d-1})$
    being the product of the Borel $\sigma$-algebras.
    The definition of the probability measure $\PP$
    follows the Markov chain rule 
    using the Markov kernel describing the conditional distribution 
    of $(w,x)$ given $(u,v)$
    \begin{equation}
        \label{eq:markov_kernel}
        K_{(w,x)\mid(u,v)}((u,v),\cdot) = (T_{uv})_{\#}(\sigma_{\sphere^{m-1}}\otimes\sigma_{\sphere^{d-1}}),
    \end{equation}
    where $T_{uv} : \sphere^{m-1}\times\sphere^{d-1} \to (\sphere^{m-1} \cap \{u\}^\perp)\times(\sphere^{d-1}\cap\{v\}^\perp)$
    via
    \begin{align*}
        (z,y) \mapsto \bigl(\tfrac{P_{\{u\}^\perp}z}{\norm{P_{\{u\}^\perp}z}}, \tfrac{P_{\{v\}^\perp}y}{\norm{P_{\{v\}^\perp}y}}\bigr)
        \coloneqq \bigl(\tfrac{(I_m - uu^*)z}{\norm{(I_m - uu^*)z}}, \tfrac{(I_d - vv^*)y}{\norm{(I_d - vv^*)y}}\bigr).
    \end{align*}
\end{remark}}


\adj{\subsection{On the assumptions for the optimal step size}
\label{sec:as_negative_cases}

Now, we give some statements characterizing the assumptions 
in Proposition~\ref{prop:first_order_optimality}
and hence Proposition~\ref{prop:one_reduced_step_size}.}

\begin{remark}[Almost sure non-zero and monotone sequence of objective values]
    \label{rem:almost_sure_non_zero_objective}
    Under the assumption that $A\neq V$
    we have by construction 
    that the initial objective fulfills $a_0 = \scp{u^0}{(A - V) v^0} > 0$ almost surely.
    Therewith,
    under the assumptions from Proposition~\ref{prop:first_order_optimality},
    for every iterate it holds a.s. that
    \begin{align*}
        a_{k+1}^2 = q_k^2(\tau_k) > a_k^2 > a_0^2 > 0,\quad\text{and}\quad a_{k+1} = q_{k}(\tau_{k}) > a_{k}> a_{0}>0.
    \end{align*}
    Especially we have throughout the iterations 
    that $v^k \not \in \ker(A - V)$
    and $u^k \not \in \ker((A - V)^*)$.
\end{remark}

In the following we will use the abbreviation
\begin{align}
\label{eq:AV}
\AV\coloneqq \begin{pmatrix}0 & (A-V)^* \\ (A-V) & 0 \end{pmatrix}.
\end{align}

The next result is similar to \cite[Lem 2.8]{BLSW24} and \adj{guarantees that 
$b_k + c_k \neq 0$} in Algorithm~\ref{alg:OpNorm3} almost surely if $(u^{k},v^{k}) \in \sphere^{m-1}\times\sphere^{d-1}$ 
is not a pair of left  and right singular vectors.

\begin{lemma}
    \label{lem:update_step_as}
    Assume that $(u^k,v^k) \in \sphere^{d-1}\times\sphere^{m-1}$ 
    is not a pair of left  and right singular vectors \adj{of $A-V$} corresponding 
    to a common singular value, 
    and that $x^k$ and $w^k$ are \adj{sampled} according to steps 3 and 4 
    of Algorithm~\ref{alg:OpNorm3}. 
    Then it holds that $b_k + c_k \neq 0$ almost surely.
\end{lemma}
\begin{proof}
Let 
\begin{align*}
    0 
    = b_k + c_k 
    = \scp{\AV\begin{pmatrix} v^k \\ u^k \end{pmatrix}}{\begin{pmatrix} x^k \\ w^k \end{pmatrix}}.
\end{align*}
Accordingly, 
\adj{we have $x^k \in \{v^k\}^\perp$ and $w^k \in \{u^k\}^\perp$,
or equivalently in $\RR^{d+m}$}
\begin{align*}
    \adj{\left[\begin{pmatrix} v^k \\ 0 \end{pmatrix} \perp \begin{pmatrix} x^k \\ 0 \end{pmatrix}  
    \land 
    \begin{pmatrix} 0 \\ u^k \end{pmatrix} \perp \begin{pmatrix} 0 \\ w^k \end{pmatrix} \right]}
    \quad\text{and}\quad 
    \left[\begin{pmatrix} x^k \\ w^k \end{pmatrix} \perp \begin{pmatrix} (A-V)^* u^k \\ (A-V) v^k \end{pmatrix}\right].
\end{align*}
\adj{Those conditions do not restrict the dimensions 
of the sampling space, if it was true that
\begin{equation}
    \label{eq:1}
    \begin{pmatrix} (A-V)^* u^k \\ (A-V) v^k \end{pmatrix} 
    = \lambda_1 \begin{pmatrix} v^k \\ 0 \end{pmatrix}
    + \lambda_2 \begin{pmatrix} 0 \\ u^k \end{pmatrix}
\end{equation}
for some $\lambda_1, \lambda_2 \neq 0$.
From \eqref{eq:1} we conclude that 
\begin{equation*}
    \lambda_1 = \lambda_1 \scp{v^k}{v^k}
    = \scp{v^k}{(A - V)^* u^k}
    = \scp{(A - V)v^k}{ u^k}
    = \lambda_2\scp{ u^k}{u^k}
    = \lambda_2.
\end{equation*}}
If it was true that
\begin{align*}
    \begin{pmatrix} v^k \\ u^k \end{pmatrix} 
    = \lambda \begin{pmatrix} (A-V)^* u^k \\ (A-V) v^k \end{pmatrix}
\end{align*}
for some $\lambda \neq 0$, 
then we would have
\begin{align*}
    \left\{ \begin{array}{l} 
    v^{k} = \lambda^2 (A-V)^* (A-V) v^k, \\ 
    u^{k} = \lambda^2 (A-V)(A-V)^* u^k.
    \quad \Leftrightarrow \quad 
    \AV \begin{bmatrix}
        u^k \\ v^k
    \end{bmatrix} 
    = \sigma \begin{bmatrix}
        u^k \\ v^k
    \end{bmatrix} 
\end{array}\right.
\end{align*}
Thus, $(u^k,v^k)$ would be a pair of left  and right singular vectors corresponding to a common singular value, contradicting the initial assumption. Therefore, the random vector $(x^k,w^k)^*$ must lie in a subspace of dimension at most $(d-1)+(m-1)-1$. 
Since both $x^k \in \sphere^{d-1}$ and $w^k \in \sphere^{m-1}$ 
are independently and uniformly drawn, 
this condition restricts the random vector $(x^k,w^k)^*$ to a set of zero measure, 
\adj{see Remark~\ref{rem:distribution-x-w} and~\ref{rem:distribution-x-w-const}}.
Consequently, this situation almost surely never occurs.
\end{proof}

The next result is similar to \cite[Lem 2.8]{BLSW24}
and Lemma~\ref{lem:update_step_as} but more technical
\adj{and ensures that the iterations in Algorithm~\ref{alg:OpNorm3} are well-defined (a.s.) 
as long as $(u^k, v^k) \in \sphere^{m-1}\times\sphere^{d-1}$ is not a pair of singular vectors of $A - V$}.

\begin{proposition}
    \label{prop:factos_alomst_surely_not_zero}
    Assume that $(u^k, v^k) \in \sphere^{m-1}\times\sphere^{d-1}$ is not a pair of singular vectors
    of $A - V$,
    and that $A-V\neq 0$.
    Furthermore, let $a_k, b_k ,c_k$ and $d_k$ be generated as in Algorithm~\ref{alg:OpNorm3}.
    Then it almost surely holds that
    \begin{align*}
        a_k b_k + c_k d_k \neq 0, \quad 
        a_k c_k + b_k d_k \neq 0\quad 
        \text{and}\quad 
        b_k c_k - a_k d_k \neq 0.
    \end{align*}
\end{proposition}
\begin{proof}
Under the assumption that $A - V \neq 0$,
we already have by Remark~\ref{rem:almost_sure_non_zero_objective},
that $a_k = \scp{u^k}{(A - V) v^k} \neq 0$ almost surely.
To show the assertion
we rewrite the first factor by definition from \eqref{eq:q_k_abcd_k}
as follows 
\begin{align}
    \label{eq:abcd_parameter}
    a_k b_k + c_k d_k
    = \scp{\begin{pmatrix}
        w^k \\ x^k 
    \end{pmatrix}}{\begin{pmatrix}
        (A - V) v^k \scp{(A - V)v^k}{u^k} \\
        (A - V)^* u^k \scp{(A - V)x^k}{w^k}
    \end{pmatrix}}.
\end{align}
If we assume $a_k b_k + c_k d_k = 0$
in combination with \adj{$v^{k}\perp x^{k}$ and $u^{k}\perp w^{k}$}
a similar discussion 
of proper subspaces as in Lemma~\ref{lem:update_step_as}
with some $\eta \in \RR$,
leads to 
\begin{align}
    \label{eq:aux-eq-parameter-abcd}
    \left\{ \begin{array}{l}
    \eta v^k = (A - V)^* u^k \scp{(A - V)x^k}{w^k} = (A - V)^* u^k d_k,  \\
    \eta u^k = (A - V) v^k \scp{(A - V)v^k}{u^k} = (A - V) v^k a_k.
    \end{array}\right.
\end{align}
Multiplying the last equation in~\eqref{eq:aux-eq-parameter-abcd} 
with $(u^k)^*$ from the left yields by definition from \eqref{eq:q_k_abcd_k} that
\begin{align}
    \label{eq:eta_a_k}
    \eta = \eta \scp{u^k}{u^k}
    = \scp{(A - V)v^k}{u^k}^2
    = a_k^2
\end{align}
and with $(w^k)^*$ yields 
\begin{align*}
    0 & = \scp{(A - V)v^k}{w^k}\scp{(A - V) v^k}{u^k} \\
    \quad\Rightarrow\quad
    0 & = \scp{(A - V)v^k}{w^k} \; \lor \; 0 = \scp{(A - V) v^k}{u^k}.
\end{align*}
Multiplying the first equation in~\eqref{eq:aux-eq-parameter-abcd} with $(x^k)^*$ from the left yields
\begin{align*}
    0 & = \scp{(A - V)x^k}{u^k}\scp{(A - V) x^k}{w^k} \\
    \quad\Rightarrow\quad
    0 & = \scp{(A - V)x^k}{u^k} \; \lor \; 0 = \scp{(A - V) x^k}{w^k}.
\end{align*}
From the latter,
we end up with four cases:
three of them result into
$\eta = a_k^2 = \scp{u^k}{(A - V) v^k}^2 = 0$,
and thus $a_k = 0$ from \eqref{eq:aux-eq-parameter-abcd}
and utilizing \eqref{eq:eta_a_k},
where the last case 
($0 = \scp{(A - V)v^k}{w^k} = b_k$ and $0 = \scp{(A - V)x^k}{u^k} = c_k$)
implies that $(u^k, v^k)$ is a pair of singular vectors%
---all contradicting the assumptions,
since $a_k \neq 0$ using Remark~\ref{rem:almost_sure_non_zero_objective}
and $(u^k, v^k)$ is not a pair of singular vectors,
utilizing again a similar argumentation as in Lemma~\ref{lem:update_step_as}.

A similar construction, argumentation and conclusion 
can be done for the second factor,
since 
\begin{align*}
    a_k c_k + b_k d_k
    = 
    \scp{\begin{pmatrix}
        w^k \\ x^k 
    \end{pmatrix}}{\begin{pmatrix}
        (A - V) v^k \scp{(A - V)x^k}{w^k} \\
        (A - V)^* u^k \scp{(A - V)v^k}{u^k}
    \end{pmatrix}}
\end{align*}

For the last one,
we obtain
\begin{align*}
    b_k c_k - a_k d_k
    = 
    \scp{\begin{pmatrix}
        w^k \\ x^k 
    \end{pmatrix}}{\begin{pmatrix}
        - (A - V) x^k \scp{(A - V)v^k}{u^k} \\
        (A - V)^* u^k \scp{(A - V)v^k}{w^k}
    \end{pmatrix}}.
\end{align*}
Hence, 
if we assume that $b_k c_k - a_k d_k = 0$,
we have as previously some $\eta \in \RR$,
since $w^k \in \sphere^{m-1} \cap \{u^k\}^\perp$
and $x^k \in \sphere^{d-1} \cap \{v^k\}^\perp$ are randomly \adj{chosen},
such that 
\begin{align}
    \label{eq:aux-eq-parameter-acbd}
    \left\{ \begin{array}{l}
    \eta u^k = - (A - V) x^k a_k = - (A - V) x^k \scp{(A - V)v^k}{u^k}, \\
    \eta v^k = (A - V)^* u^k b_k = (A - V)^* u^k \scp{(A - V)v^k}{w^k}.
    \end{array}\right.
\end{align}
Taking the inner product with $w^{k}$ in the first equation in \eqref{eq:aux-eq-parameter-acbd} 
and with $x^{k}$ in the second 
yields $0 = b_k c_k = a_k d_k$.
Similarly, the inner products with $u^{k}$ and $v^{k}$ in the first and second equation in \eqref{eq:aux-eq-parameter-acbd}, respectively,
yields $\eta = a_k b_k = a_k c_k$.
Hence, we obtain $d_k = 0$, since $a_k = \scp{u^k}{(A - V) v^k} \neq 0$
and we conclude $c_k = 0$ or $b_k = 0$
and consequently $\eta = 0$.
Now, taking the inner product with $u^{k}$ in the first and with $v^{k}$ in the second equation of \eqref{eq:aux-eq-parameter-acbd}, respectively, yields
\begin{align*}
    0 & = \eta \norm{u^k}^2 = -\scp{u^k}{(A - V) x^k}\scp{u^k}{(A - V) v^k} = - c_k a_k, \\
    0 & = \eta \norm{v^k}^2 = \scp{u^k}{(A - V) v^k}\scp{w^k}{(A - V) v^k} = a_k b_k,
\end{align*}
since $\eta = 0$.
From the first equation we obtain that $a_k = 0$ or $c_k = 0$,
where $a_k = 0$ contradicts the assumptions
and hence \adj{it holds}  $c_k = 0$, 
and similarly we get $b_k = 0$ from the second equation.
As in the first case,
we end up with a contradiction,
since $c_k = 0$ and $b_k = 0$ would imply that $(u^k, v^k)$ is a pair of singular vectors, see Lemma~\ref{lem:update_step_as}.
\end{proof}

\begin{remark}[$d_k$ is zero with probability zero]
    Notice that we almost surely have $d_k = \scp{w^k}{(A - V) x^k} \neq 0$,
    if $A - V\neq 0$
    and $(u^k, v^k)$ is not a pair of singular vectors to a common singular value.
    One can see this as follows:
    \begin{align*}
        d_k = 0 
        & \iff
        \scp{z^k - \scp{y^k}{z^k}y^k}{(A - V) (y^k - \scp{y^k}{v^k}v^k)} = 0 \\
        & \iff
        (A - V)(I_d - v^k(v^k)^*) y^k \in \langle(I_m - y^k(y^k)^*) z^k \rangle^\perp.
    \end{align*}
    Since $y^k \sim \cN(0, I_d)$ and $z^k \sim \cN(0, I_m)$ and are independent, Fubini's theorem shows that the probability for the latter inclusion to hold is zero similarly as in Remark~\ref{rem:drop_assumptions_in_algo_1}.
\end{remark}

We now show 
that if $(u^k, v^k)$ is not a pair of left and right singular vec\-tors 
corresponding to a common singular \adj{value}, 
then the same is true for $(u^{k+1}, v^{k+1})$ almost surely.
\adj{Setting $L = A - V$ in the definition of $\tilde{E}_{\sigma}$ in \eqref{eqn:E_sigma}, we obtain equivalently} 
\begin{align}
    \label{eqn:E_sigma_special}
    \tilde E_\sigma 
    &= \left\{ \begin{bmatrix} u \\ v \end{bmatrix} \in \RR^{m+d} 
    \;\middle|\; 
    \AV \begin{bmatrix} u \\ v \end{bmatrix} 
    = \sigma \begin{bmatrix} u \\ v \end{bmatrix} \right\} \notag \\
    &= 
    \left\{ \left[\begin{matrix}
        u \\ v
    \end{matrix}\right] \in \RR^{m + d} 
    \; \middle| \;
    \begin{array}{l}
        (A - V)^*(A - V)v = \sigma^2 v, \\
        (A - V)(A - V)^*u = \sigma^2 u
    \end{array}
    \right\}.
\end{align}

\begin{lemma}
    \label{lem:normdifference-dimension-directions}
    Let \(\sigma\) be a singular value of \(A - V \in \RR^{m \times d}\) 
    with multiplicity \(r < \max\{d,m\}\), 
    and \((u, v)^* \in \RR^{m+d}\) be not a pair of left and right singular vectors 
    corresponding to \(\sigma\). 
    Then the set of directions \((w, x)^* \in \{u\}^{\perp} \times \{v\}^{\perp} \)
    for which \((u + w, v + x) \in \tilde E_\sigma\) holds 
    \adj{is} an affine subspace with dimension of $r$, \adj{$r-1$, or $r-2$ 
    or is empty}.
\end{lemma}
\begin{proof}
    The proof is similar to~\cite[Lemma~2.9]{BLSW24},
    cf.~Appendix~\ref{sec:proof-normdifference-dimension-directions}.
\end{proof}

\begin{corollary}
    \label{cor:normdistance-2-no-sing-vectors-during-iteration}
    Let \adj{$\max\{m,d\} > 2$ and $A, V \in \RR^{m\times d}$} with $A \neq V$
    and suppose the multiplicity of all singular values of \(A - V\) 
    is less than \(\max\{d, m\} - 1\).
    Then, it is almost surely the case that 
    \begin{align*}
        a_k \neq 0, \quad 
        a_k b_k + c_k d_k \neq 0, \quad 
        a_k c_k + b_k d_k \neq 0\quad 
        \text{and}\quad 
        b_k c_k - a_k d_k \neq 0.
    \end{align*}
    holds throughout all iterations of Algorithm~\ref{alg:OpNorm3} 
    if \((u^0, v^0)\) is not a pair of left and right singular vectors 
    corresponding to a common singular value of \(A - V\).
\end{corollary}
\begin{proof}
    From Remark~\ref{rem:almost_sure_non_zero_objective} we know $a_k\neq 0 $.
    For the rest we refer to Appendix~\ref{proof:normdistance-no-sing-vectors-during-iteration}
    and utilize Proposition~\ref{prop:factos_alomst_surely_not_zero}.
\end{proof}

\adj{The restriction of the dimensions $\max\{m,d\}>2$ 
is only technical:
If $m = d = 2$, then,
as in \cite{BLSW24},
a simple geometric discussion leads us to the optimal pair of singular vectors 
within one iteration and Algorithm~\ref{alg:OpNorm3} will stop.
If $\min\{m,d\} = 1$, then there is one positive singular value 
and Algorithm~\ref{alg:OpNorm3} will stop with the optimal value in the initial step.}


\begin{remark}[Algorithm~\ref{alg:OpNorm3} is a stochastic projected gradient method]
  \label{rem:stoch_gradient_ascent_op2}
  \adj{Despite being a random search method, the method can still be interpreted as a stochastic gradient method due to the choice of step sizes:}
  Since the step size $\tau_{k}$ is positive for 
  \begin{align}
    \label{eq:double-half_sphere}
    a_k b_k + c_k d_k
    = \scp{\begin{pmatrix}
        w^k \\ x^k 
    \end{pmatrix}}{\begin{pmatrix}
        (A - V) v^k \scp{(A - V)v^k}{u^k} \\
        (A - V)^* u^k \scp{(A - V)x^k}{w^k}
    \end{pmatrix}}>0,
  \end{align}
  and negative for $a_k b_k + c_k d_k<0$, and thus, we effectively move in direction 
  \[
    \sign(a_k b_k + c_k d_k)\begin{pmatrix}
        x^k \\ w_{k}
    \end{pmatrix}.
  \]
  This means that the direction in which the algorithm 
  moves is uniformly dis\-tri\-bu\-ted on the product of two half spheres 
  in the space orthogonal to $(u^k, v^{k})^*$ 
  and defined by \eqref{eq:double-half_sphere}.
  Due to symmetry of the distribution, 
  the expected direction $(\EE(\sign(a_{k})x^{k}), \EE(\sign(a_k b_k + c_k d_k)w^{k}))^*$ 
  is a positive multiple of the projected version of $\AV (v^k, u^k)^*$ 
  on the space orthogonal to $(u^k,v^{k})$. 
  We conclude that Algorithm~\ref{alg:OpNorm3} 
  is a stochastic ascent algorithm in the sense that the direction 
  is a multiple of the projection 
  on the gradient $\AV (v^k, u^k)^*$ of the objective $\scp{u}{(A - V) v}$.
\end{remark}

We collect all our assumptions.
Notably,
if we choose $(u^0, v^0)$ uniformly on $\sphere^{m-1} \times \sphere^{d-1}$,
then the assumptions in Corollary~\ref{cor:normdistance-2-no-sing-vectors-during-iteration} guarantee 
that $(u^0, v^0)$ is not a pair of singular vectors to a common singular value of $A - V$,
and Corollary~\ref{cor:normdistance-2-no-sing-vectors-during-iteration}
guarantees that the sequences generated in Algorithm~\ref{alg:OpNorm3}
are well defined.
Hence,
those assumptions are sufficient to run the algorithm.

\begin{assumption} \mbox{}
    \label{ass:assump_2}
    Let $0 \neq A - V \in \RR^{m \times d}$
    \adj{have} no singular value with multiplicity larger or equal than $\max\{m,d\} - 1$.
\end{assumption}

From now on,
as long as nothing else is said,
we assume that Assumption~\ref{ass:assump_2} holds true.
However,
we will show in the end
that even the Assumption~\ref{ass:assump_2} can be dropped,
see Remark~\ref{rem:drop_assumptions_in_algo_1}.
Assumption~\ref{ass:assump_2} is equivalently fulfilled,
if $A - V$  has singular values with multiplicity at most $\max\{m,d\}-2$.

\begin{corollary}
    \label{cor:monotonicity_abs_ak_in_OpNorm3}
    The sequence $(a_k)_{k \in \N}$
    generated by Algorithm~\ref{alg:OpNorm3}
    is strictly increasing and almost surely convergent.
\end{corollary}
\begin{proof}
    Since the stepsizes $\tau_k$ and $\xi_k$ are well defined \adj{almost surely}
    and the method is initialized with $a_{0}>0$,
    the sequence $(a_k)_{k \in \N}$ is increasing,
    \adj{cf.~Remark~\ref{rem:almost_sure_non_zero_objective}}. 
    Furthermore, 
    by construction,
    we have $a_{k+1} = q_k(\tau_k) \leq \norm{A - V}$,
    which shows boundedness of the sequence $(a_{k})_{k \in \N}$.
\end{proof}

\subsection{Almost sure convergence to $\norm{A-V}$}
\label{sec:convergence_as_two}

In this subsection,
we give the proof for the almost sure convergence of the sequence 
of objective values generated by Algorithm~\ref{alg:OpNorm3} to the largest singular value.
Additionally, we prove that 
the sequence of the vectors $(u^k, v^k)_{k \in \N}$
converges to a pair of singular vectors corresponding to $\sigma_{1} = \norm{A-V}$
\adj{if the corresponding space is one-dimensional 
and otherwise we show the almost sure convergence
of the random variables $(u^k, v^k)$
towards the leading singular vectors space $\tilde E_{\sigma_1}$.
}

\begin{lemma}
    \label{lem:asscent_lem_2}
    \adj{Let $(u^k, v^k)_{k \in \NN}$ and $(a_k)_{k \in \NN}, (b_k)_{k \in \NN}, (c_k)_{k \in \NN}$ 
    and $(d_k)_{k \in \NN}$ be the sequences of random variables generated} from Algorithm~\ref{alg:OpNorm3} 
    with step sizes $(\tau^k, \xi^k)_{k \in \NN}$ from Proposition~\ref{prop:one_reduced_step_size}, 
    i.e. \eqref{eqn:opt_crit_two_stepsizes_solved},
    it holds
    \begin{align*}
        & \scp{u^{k+1}}{(A - V) v^{k+1}}^2 - \scp{u^k}{(A - V) v^k}^2 \\
        & = c_k^2 + \tau_k (a_k b_k + c_k d_k) \\
        & = b_k^2 + \xi_k (a_kc_k + b_k d_k) \\
        & = \frac{1}{2}(b_k^2 + c_k^2 + \tau_k (a_k b_k + c_kd_k) + \xi_k(a_kc_k + b_k d_k)).
    \end{align*}
\end{lemma}
\begin{proof}
    From Remark~\ref{rem:rewritten_squared_objective},
    and by using the optimality criteria from Lemma~\ref{lem:two_step_sizes} for $\xi_k$,
    we have
    \begin{align*}
        q^2_k(\tau_k,\xi_k)
        & = (a_k + \xi_k c_k)(a_k + \tau_k b_k) \\
        & = \left(a_k + \frac{c_k + \tau_k d_k}{a_k + \tau_k b_k}c_k\right)
        (a_k + \tau_k b_k) \\
        & = a_k^2 + c_k^2 + \tau_k (a_k b_k + c_k d_k).
    \end{align*}
    We use
    $q^2_k(\tau_k,\xi_k) = a_{k+1}^2 = \scp{u^{k+1}}{(A - V) v^{k+1}}^2$
    and $a_k^2 = \scp{u^k}{(A - V) v^k}^2$ and get the first assertion.
    The second assertion 
    follows from the rewritten objective 
    by replacing instead $\tau_k = \tfrac{b_k + \xi_k d_k}{a_k + \xi_k c_k}$.
    The last assertion is just a combination of the first and second one.
\end{proof}

Next,
we can give a proof of almost sure convergence of some of the sequences 
given by some of the parameters $\{a_k, b_k, c_k, d_k\}_{k \in \N}$
generated and used in the algorithms.

\begin{corollary}
    \label{cor:c_k_to_0}
    \adj{Let $(a_k)_{k \in \NN}, (b_k)_{k \in \NN}, (c_k)_{k \in \NN}$ and $(d_k)_{k \in \NN}$ 
    be the sequences of random variables generated} by Algorithm~\ref{alg:OpNorm3}.
    Then, it holds $c_k\to 0$ and $\tau_k(a_k b_k + c_k d_k) \to 0$
    for $k \to \infty$ almost surely.
\end{corollary}
\begin{proof}
    By Lemma~\ref{lem:asscent_lem_2},
    the optimal choice of $\tau_k$ from Lemma~\ref{prop:one_reduced_step_size}
    and the fact that the optimal $\tau_k$ and $a_k b_k + c_k d_k$
    have the same sign by Proposition~\ref{prop:one_reduced_step_size},
    we have $\tau_k( a_k b_k + c_k d_k) \geq 0$
    and therefore conclude 
    \begin{align*}
        0 
        \leq 
        c_k^2 + \tau_k( a_k b_k + c_k d_k)
        = q^2_k(\tau_k) - a_k^2
        = a_{k+1}^2 - a_k^2
        \to 0,
        \quad k \to \infty
    \end{align*}
    since $a^2_k$ is convergent by Corollary~\ref{cor:monotonicity_abs_ak_in_OpNorm3}.
\end{proof}

\begin{remark}[Solving the ascent via $\xi$ instead]
    \label{rem:b_k_to_0}
    Reducing our 2-dimensional objective into a 1-dimensional objective $q_k(\tau, \xi)$,
    can be also done for $\xi$, 
    by plugging in the optimality criteria from Lemma~\ref{lem:two_step_sizes},
    which yields a similar result as in Lemma~\ref{prop:one_reduced_step_size}.
    From that we also obtain $b_k\to 0$ and $0 \leq \xi_k(a_k c_k + b_k d_k) \to 0$
    for $k \to \infty$ almost surely.
\end{remark}

\begin{corollary}
    \label{cor:ab+cd_and_b_to_0}
    \adj{Let $(a_k)_{k \in \NN}, (b_k)_{k \in \NN}, (c_k)_{k \in \NN}$ and $(d_k)_{k \in \NN}$ 
    be the sequences of random variables generated} by Algorithm~\ref{alg:OpNorm3}.
    Then, it holds $a_k b_k + c_k d_k  \to 0$ and $b_k\to 0$
    for $k \to \infty$ almost surely.
\end{corollary}
\begin{proof}
    (Similar to the proof of \cite[Lem.~2.13]{BLSW24}.)
    In more detail,
    we have that $\tau_k$ is a solution of $q_k^2(\tau_k)=0$ in \eqref{eq:root_tau}.
    This means that
    \begin{align*}
        a_k b_k + c_k d_k 
        & = \tau_k (a_k^2 - b_k^2 + c^2_k - d_k^2)
        + \tau_k^2(a_k b_k + c_k d_k) \\
        & = \tau_k(a_k^2 - b_k^2 + c^2_k - d_k^2 + \tau_k (a_k b_k + c_k d_k))
    \end{align*}
    such that 
    \begin{align*}
        (a_k b_k + c_k d_k)^2
        = \tau_k(a_k b_k + c_k d_k)
        (a_k^2 - b_k^2 + c^2_k - d_k^2 + \tau_k (a_k b_k + c_k d_k)).
    \end{align*}
    Since we have $\tau_k (a_k b_k + c_k d_k) \to 0$
    for $k \to \infty$ almost surely by~Corollary~\ref{cor:c_k_to_0},
    and $a_k^2 - b_k^2 + c^2_k - d_k^2$
    is bounded by $4\norm{A - V}^2$,
    we have $a_kb_k + c_kd_k \to 0$ for $k \to \infty$ almost surely. Now as $a_k b_k + c_k d_k \to 0$ and $c_k \to 0$ for $k \to \infty$ a.s.,
    we also have $a_kb_k \to 0$ for $k \to \infty$ a.s.
    Since we have $\norm{A - V} \geq a_k \geq a_0 > 0$ under Assumption~\ref{ass:assump_2},
    we necessarily have $b_k \to 0$ for $k \to \infty$ a.s.
\end{proof}

Finally,
we have everything to prove the almost sure convergence 
of the se\-quence $(a_k)_{k \in \N}$ induced by $((u^k, v^k))_{k \in \N}$ 
from Algorithm~\ref{alg:OpNorm3} to a singular value.

%

\begin{lemma}
    \label{lem:E_c_sqr}
    \adj{Let $(u^k, v^k)_{k \in \NN}$ and $(b_k)_{k \in \NN}, (c_k)_{k \in \NN}$ 
    be the sequences of random variables 
    generated by Algorithm~\ref{alg:OpNorm3}.
    Then, it holds
    \begin{align*}
        \EE\left[ c_k^2 \mid u^k, v^k \right]
            = \tfrac{1}{d-1}
            \left\| \left(I_d-v^k(v^k)^*\right)(A-V)^*u^k \right\|^2\\
        \EE\left[ b_k^2 \mid u^k, v^k \right]
        = \tfrac{1}{m-1}
        \left\| \left(I_m-u^k(u^k)^*\right)(A-V)v^k \right\|^2.
    \end{align*}   } 
\end{lemma}
\begin{proof}
    \adj{Rewrite the parameters $c_k^2$ and $b_k^2$
    from \eqref{eq:q_k_abcd_k}
    and use \cite[Lem.~2.14]{BLSW24}.} 
\end{proof}

\adj{For the next result 
we introduce the distance of a set and a vector by 
\begin{align}
    \label{eq:dist_sets}
    \textup{dist}(C, x) \coloneqq \inf_{y \in C} \norm{x - y}.
\end{align}
This allows to prove the almost sure convergence 
of a subsequence generated by Algorithm~\ref{alg:OpNorm3} 
into a singular vector space.
}

\begin{proposition}
    \label{prop:acc_point_opt_pair_two_stepsizes}
    \adj{Let $(u^k, v^k)_{k \in \NN}$ be the sequence of random variables 
    generated by Algorithm~\ref{alg:OpNorm3}.
    Then,
    there exists a subsequence $(u^{k_j}, v^{k_j})_{j \in \NN}$ 
    and a real-valued ranodom variable $\sigma$
    with values in the set of singular values of $A - V$ 
    such that 
    \[
        \textup{dist}(\tilde E_{\sigma}, (u^{k_j}, v^{k_j})) \to 0, \quad j \to \infty
        \quad \textrm{and} \quad
        \abs{\scp{v^k}{(A - V) v^k}} \nearrow \sigma, \quad k \to \infty.
    \]}
\end{proposition}
\begin{proof}
    \adj{We adapt the strategy of the proof \cite[Prop.~2.15]{BLSW24}.
    According to Corollary~\ref{cor:c_k_to_0} and Corollary~\ref{cor:ab+cd_and_b_to_0}, 
    we know the almost sure convergence of the sequences $(b_k)_{k \in \N}$
    and $(c_k)_{k \in \N}$ to zero,
    and $b_k^2, c_k^2 \leq \norm{A - V}^2$.
    Then,
    Lebesgue's dominated convergence theorem shows 
    that $\EE[b_k^2]\to 0$ and $\EE[c_k^2]\to 0$ for $k \to \infty$. 
    Hence, 
    by Lemma~\ref{lem:E_c_sqr} we obtain 
    \begin{align*}
        \frac{1}{d-1} \EE\Big[ \|\big(I_d-v^k(v^k)^*\big)(A-V)^*u^k\|_2^2 \Big] 
        = 
        \EE\Big[\EE\big[ c_k^2 \mid u^k, v^k \big]\Big]
        =
        \EE\big[ c_k^2\big] \to 0,
    \end{align*}
    and 
    \begin{align*}
        \frac{1}{m-1} \EE\Big[ \|\big(I_d-u^k(u^k)^*\big)(A-V)v^k\|_2^2 \Big] 
        = 
        \EE\Big[\EE\big[ b_k^2 \mid u^k, v^k \big]\Big]
        =
        \EE\big[ b_k^2\big] \to 0,
    \end{align*}
    for $k \to \infty$,.
    The latter equations 
    imply a subsequence $(u,v)^{k_j}$ of $(u,v)^k$
    such that
    \begin{align*}
        \begin{cases}
            \|\big(I_d-u^{k_j}(u^{k_j})^*\big)(A-V)v^{k_j}\|_2^2 \to 0, \\
            \|\big(I_d-v^{k_j}(v^{k_j})^*\big)(A-V)^*u^{k_j}\|_2^2 \to 0
        \end{cases}
        \quad j \to \infty
        \quad \textrm{a.s.} 
    \end{align*}
    This in turn shows that 
    \begin{align*}
        \begin{cases}
            \textrm{dist}(\ker(I_d-u^{k_j}(u^{k_j})^*), (A-V)v^{k_j}) \to 0, \\
            \textrm{dist}(\ker(I_d-v^{k_j}(v^{k_j})^*), (A-V)^*u^{k_j}) \to 0
        \end{cases}
        \quad j \to \infty
        \quad \textrm{a.s.} 
    \end{align*}
    and hence 
    \begin{align*}
        \begin{cases}
            \textrm{dist}(\textrm{span}(u^{k_j}), (A-V)v^{k_j}) \to 0, \\
            \textrm{dist}(\textrm{span}(v^{k_j}), (A-V)^*u^{k_j})  \to 0
        \end{cases}
        \quad j \to \infty
        \quad \textrm{a.s.} 
    \end{align*}
    Therewith,
    two sequences of random variables $(\sigma_j)_{j \in \NN}$ and $(\eta_j)_{j \in \NN}$ exists
    such that 
    \begin{align*}
        \begin{cases}
            \norm{(A-V) v^{k_j} - \sigma_j  u^{k_j}}  \to 0,\\
            \norm{(A-V)^* u^{k_j} - \eta_j  v^{k_j}} \to 0
        \end{cases}
        \quad j \to \infty
        \quad \textrm{a.e.} \; w \in \Omega.
    \end{align*}
    By the almost sure convergence of the objective value $a_k = \scp{u^k}{(A - V) v^k}$,
    we have that $\lim_{j \to \infty} \sigma_j = \lim_{j \to \infty} \eta_j = \sigma$ almost surely.
    This shows
    that $\sigma$ takes values in the set of singular values of $A - V$, cf.~\eqref{eq:1},
    yielding the first assertion.
    The convergence of the entire sequence 
    to the corresponding singular value
    follows from the monotonicity,
    see Remark~\ref{rem:almost_sure_non_zero_objective},
    since 
    \begin{align*}
        \abs{\scp{u^{k_j}}{(A - V) v^{k_j}}}
        \leq \abs{\scp{u^k}{(A - V) v^k}}
        \leq \abs{\scp{u^{k_{j+1}}}{(A - V) v^{k_{j+1}}}}
    \end{align*}
    for all $k \in \{k_j,...,k_{j+1}\}$,
    almost surely.}
\end{proof}

The next statement will help to show that \adj{$(a_{k})_{k \in \NN}$} does indeed converge 
to the maximal singular value, i.e. the operator norm of $A - V$.

\begin{lemma} 
    \label{lem:normdifference-B_Lemma}
    \adj{Let $(u^k, v^k)$ be the sequence of random variables 
    generated by Algorithm~\ref{alg:OpNorm3}.}
    For the set
    \begin{align*} 
        B \coloneqq \{(u,v)\in\sphere^{m-1}\times\sphere^{d-1}: \abs{\scp{u}{(A-V)v}} > \sigma_2(A-V)\}
    \end{align*}
    the following holds:
    \begin{enumerate}[i)]
        \item
        If $(u^k,v^{k})\in B$ holds almost surely, 
        then $v^{k+1}\in B$ holds almost surely as well.
        \item
        If $(u^{k_0}, v^{k_0})\in B$ holds for some $k_0 \in \NN$ almost surely, 
        then almost surely holds 
        \[\lim_{k\to\infty} \abs{\scp{u^k}{(A-V)v^k}}
        = \sigma_1(A-V) = \norm{A-V}.\]
        \adj{Moreover, it holds $\textup{dist}(\tilde E_{\sigma_1}, (u^k, v^k)) \to 0$
        for $k \to \infty$ almost surely.}
    \end{enumerate}
\end{lemma}
\begin{proof}
    Utilizing Proposition~\ref{prop:acc_point_opt_pair_two_stepsizes}
    and the monotonicity of $\abs{\scp{u^k}{(A-V)v^k}}$, 
    see Remark~\ref{rem:almost_sure_non_zero_objective}, yields the assertion.
    \adj{We refer to Appendix~\ref{proof:normdifference-B_Lemma}} 
    and to \cite[Lem.~2.16]{BLSW24} as it is a similar proof technique.
\end{proof}

The next technical proposition
from \cite{BLSW24}
will help us to prove that in each step the probability
that the next iterate will fulfill 
$\abs{\scp{u^{k+1}}{(A - V)v^{k+1}}} > \sigma_{2}$
is bounded uniformly from below by a positive constant.

\begin{proposition}[cf.~{\cite[Lem.~2.17]{BLSW24}}] \label{prop:inf_probability}
    \adj{Let $\theta>0$ and $v,\tilde v\in\sphere^{d-1}$ and define the hyperspherical cap around $\tilde{v}$ with polar angle $\theta$ by
    \begin{align*}
      B_{\theta}(\tilde{v}) &\coloneqq \left\{ w\in\sphere^{d-1}\ \middle|\ \scp{w}{\tilde{v}} \geq \cos(\theta) \right\} 
    \end{align*}
    further define
    \begin{align*}
    D_v & := \left\{x\in \{v\}^{\bot} \,\cap\, \sphere^{d-1} \;\middle|\;  \frac{v+\tau x}{\norm{v+\tau x}} \in B_{\theta}(\tilde v)\cup B_{\theta}(-\tilde v) \text{ for some } \tau\in\RR\right\}.
  \end{align*}
 Then there exists a constant $p_{\theta,d}$ independent of $v$ and $\tilde{v}$ such that 
  \begin{align}\label{eq:def_p_varepsilon}
    \inf_{v\in\sphere^{d-1}} \sigma_{\sphere^{d-1} \cap \{v\}^\perp}\big( D_v \big) = p_{\theta,d}>0.
  \end{align}}
\end{proposition}

Now, to prove the convergence almost surely,
the uniform bound from Pro\-po\-si\-tion~\ref{prop:inf_probability} is crucial,
and the independence of the step sizes is necessary 
for the following proof strategy.

\begin{theorem}
    \label{thm:convergence_B}
    \adj{Let $(u^k, v^k)$ the sequence of random variables 
    generated by Algorithm~\ref{alg:OpNorm3}.}
    Then, it holds 
    $\lim_{k\to\infty}\abs{\scp{u^{k}}{(A - V) v^{k}}} = \sigma_1$ 
    almost surely. 
    More\-over, 
    if the singular vector space to $\sigma_1$ is one-dimensional, 
    then $(u^{k}, v^{k})$ converges almost surely 
    \adj{to an optimal pair in $\sphere^{m-1}\times\sphere^{d-1}$ corresponding to $\sigma_1$}.
\end{theorem}
\begin{proof}
    We argue similar to the proof of \cite[Thm.~2.19]{BLSW24}:
    \adj{We consider the set 
    $B = \{(u, v) \in \sphere^{m-1}\times \sphere^{d-1} \colon \abs{\scp{u^{k}}{(A - V) v^{k}}} > \sigma_2\}$
    from Lemma~\ref{lem:normdifference-B_Lemma}
    and define 
    \begin{align*}
        \Omega_k := \{w \in \Omega \mid (u^{k}(w), v^{k}(w)) \not\in B\}.
    \end{align*}}
    \adj{By Lemma~\ref{lem:normdifference-B_Lemma}~i)}
    we immediately have $\Omega_{k+1} \subseteq \Omega_k$ and 
    \begin{align*}
        \{w \in \Omega \mid \abs{\scp{u^{k}(w)}{(A - V) v^{k}(w)}} \not\to \sigma_1\}
        = 
        \bigcap_{k = 0}^\infty \Omega_k.
    \end{align*}
    Since $B$ is relative open
    \adj{and due to the continuity of the objective value,}
    we find \adj{$\theta, \theta' > 0$}
    such that 
    \begin{align*}
        \adj{C_{\theta,\theta'} 
        \coloneqq 
        C_{\theta'} \times C_{\theta} 
        \coloneqq
        \big(B_{\theta'}(\tfrac{(A - V) v_1}{\norm{(A - V) v_1}}) \cap \sphere^{m-1}\big)
        \times \big(B_\theta(v_1) \cap \sphere^{d-1}\big) 
        \subset B}
    \end{align*}
    for the optimal pair of singular vectors 
    $(u_1:=\tfrac{(A - V) v_1}{\norm{(A - V) v_1}}, v_1) \in B$.
    We employ now a similar inequality as in the proof of \cite[Thm.~2.19]{BLSW24},
    using \adj{the independence of the conditional distribution of $w\mid u$ and $x\mid v$
    defined by the Markov kernels $K_{w\mid u}$, respectively $K_{x\mid v}$,}
    such that 
    \begin{align}
        \label{eq:p_independance}
        \adj{\PP_{(u^{k}, v^{k}, z^{k}, y^{k})}
        = K_{(w^k,x^k)\mid(u^k,v^k)} \times \PP_{(u^{k}, v^{k})}
        = K_{w\mid u}\times K_{x\mid v} \times \PP_{(u^{k}, v^{k})}}
    \end{align}
    \adj{with kernel from Remark~\ref{rem:distribution-x-w-const}}
    and using Proposition~\ref{prop:inf_probability}
    separately for the directions $w^k$ 
    to point towards an \adj{$\theta'$}-neigh\-bor\-hood of $\tfrac{(A - V)v_1}{\norm{(A - V) v_1}}$ 
    in $\sphere^{m-1}$
    and independently $x^k$ to an \adj{$\theta$}-neigh\-bor\-hood of $v_1$ 
    in $\sphere^{d-1}$.
    In the following we prove
    \begin{align}
        \label{eq:bound_p}
        \PP(\Omega_{k+1}) 
        \leq \underbracket{(1 - p_{\adj{\theta'},m})(1 - p_{\adj{\theta}, d})}_{=: p} \PP(\Omega_k).
    \end{align}
    Then,
    the assertion follows by induction, 
    since $1 > p > 0$ is a uniform bound 
    independent of $k \in \N$, which means we have
    \begin{align*}
        \PP(\abs{\scp{u^{k}}{(A - V) v^{k}}} \not\to \sigma_1)
        \leq  \lim_{k \to \infty} \PP(A_{k})
        \leq \PP(A_0) \lim_{k \to \infty} p^k
        = 0.
    \end{align*}
    
    To show \eqref{eq:bound_p},
    we denote the update of Algorithm~\ref{alg:OpNorm3} in the output space and input space of $A-V$
    by $U_u(u^k, v^k, w^k)$ and $U_v(u^k, v^k, x^k)$, respectively,
    \adj{i.e. we have 
    $$
        U_u(u, v, w) \coloneqq \tfrac{u + \xi_{u,v} w}{\norm{u + \xi_{u,v} w}},
        \quad 
        U_v(u, v, x) \coloneqq \tfrac{v + \tau_{u,v} x}{\norm{v + \tau_{u,v} x}},
    $$
    with deterministic $\xi_{u,v}$ and $\tau_{u,v}$ once $u,v$ and $w,x$ are fixed 
    from Proposition~\ref{prop:one_reduced_step_size}.}
    \adj{We obtain by Lemma~\ref{lem:normdifference-B_Lemma} (ii) and the law of total probability 
    \begin{align*}
            \PP(\Omega_{k+1})
            &= \iint\limits_{\sphere^{m-1}\times\sphere^{d-1}} 
            \ M(u,v)  
            \ \dd \PP_{(u^k,v^k)}(u,v) \notag \\
            &= \iint\limits_{\sphere^{m-1}\times\sphere^{d-1}\setminus B}  
            \ M(u,v)
            \ \dd \PP_{(u^k,v^k)}(u,v).
    \end{align*}}
    and for any $(u,v) \in\sphere^{m-1}\times\sphere^{d-1}$ we can rewrite
    (using Fubini's lemma)
    \adj{\begin{align*}
        M(u,v) 
        &= 
        \hspace{-1.3cm}\iint\limits_{(\sphere^{m-1} \cap \{u\}^\perp)\times(\sphere^{d-1} \cap\{v\}^\perp)}
        \hspace{-1.3cm}\mathbbm{1}_{(U_u(u,v,w), U_v(u,v,x)) \not \in B}
        \ K_{(w^k,x^k)\mid (u^k,v^k)}((u,v), \ \dd w\times\dd x) \\
        &\leq 
        \hspace{-1.3cm}\iint\limits_{(\sphere^{m-1} \cap \{u\}^\perp)\times(\sphere^{d-1} \cap\{v\}^\perp)}
        \hspace{-1.3cm}\mathbbm{1}_{(U_u(u,v,w), U_v(u,v,x)) \not \in C_{\theta,\theta'}}
        \ K_{w\mid u}(u, \ \dd w) \ K_{x\mid v}(v, \ \dd x) \\
        &\leq 
        \hspace{-0.6cm}\int\limits_{\sphere^{m-1} \cap \{u\}^\perp}
        \hspace{-0.6cm}\mathbbm{1}_{U_u(u,v,w) \not \in C_{\theta'}}
        \ K_{w\mid u}(u, \ \dd w) \
        \cdot 
        \hspace{-0.6cm}\int\limits_{\sphere^{d-1} \cap\{v\}^\perp}
        \hspace{-0.6cm}\mathbbm{1}_{U_v(u,v,x) \not \in C_{\theta}}
        \ K_{x\mid v}(v, \ \dd x) \\
        &= 
        K_{w\mid u}(u, U_u^{-1}(u,v,C_{\theta'}^{c}))
        \cdot 
        K_{x\mid v}(v, U_v^{-1}(u,v,C_{\theta}^{c})).
    \end{align*}}
    Using Proposition~\ref{prop:inf_probability} 
    separately for $w \in \sphere^{m-1} \cap \{u\}^\perp$ 
    and $x \in \sphere^{d-1} \cap \{v\}^\perp$ yields
    \begin{align*}
        \PP(\Omega_{k+1})
        & \leq 
        \iint\limits_{\sphere^{m-1}\times\sphere^{d-1} \setminus B}
        (1 - p_{\adj{\theta'},m})(1 - p_{\adj{\theta},d})
        \ \mathrm d \PP_{(u^k, v^k)}(u, v) \\
        & = (1 - p_{\adj{\theta'},m})(1 - p_{\adj{\theta},d}) \PP(\Omega_k),
    \end{align*}
    \adj{using the law of total probability, again}, which proves \eqref{eq:bound_p}.
\end{proof}

\adj{\begin{remark}[Nonzero objective value and detecting null maps]
    \label{rem:OpNorm2_detecting_null_maps}
    If we take two independent samples $u^0\in \sphere^{m-1}$ and $v^0\in\sphere^{d-1}$ according to the respective uniform distribution and observe that
    \begin{align}
        \label{eqn:OpNorm2_objective_is_zero}
        \langle u^0, Av^0\rangle - \langle V^*u^0, v^0\rangle = 0,
    \end{align}
    then we can conclude that $A=V$ holds almost surely. Namely, we have equivalence of the following three assertions:
    \begin{enumerate}[i)]
        \item $A=V$, 
        \item $\PP[\langle u^0, (A-V)v^0\rangle = 0] = 1$, 
        \item $\PP[\langle u^0, (A-V)v^0\rangle = 0] >0$.
    \end{enumerate}
    Here, 
    the only interesting implication to prove is (iii) $\Rightarrow$ (i),
    which we do by proving that if $A\neq V$, 
    then $\langle u^0, (A-V)v^0\rangle = 0$ holds with probability zero: 
    For all fixed $u$ that are not contained in the null space of $(A-V)^*$ 
    we have $\langle u, (A-V)v^0\rangle = 0$ with probability zero, 
    and hence \eqref{eqn:OpNorm2_objective_is_zero} holds with probability zero by integrating over all $u$ 
    according to the uniform distribution on $\sphere^{m-1}$.
\end{remark}}

\begin{remark}[Neglecting the Assumption~\ref{ass:assump_2}] 
    \adj{\label{rem:drop_assumptions_in_algo_1}
    By Remark \ref{rem:OpNorm2_detecting_null_maps},
    we can drop the case $\norm{A-V}=0$ in Assumption~\ref{ass:assump_2}.
    If $\norm{A-V}\neq0$,
    we want to employ Lemma~\ref{lem:update_step_as} to drop Assumption~\ref{ass:assump_2}.
    Therefore,
    we inductively show that $(u^k, v^k)$ is almost surely not a pair of singular vectors 
    to a common singular value of $A - V$,
    if $\norm{A - V} \neq 0$.
    We know from Proposition~\ref{prop:optimal_pair}
    that the singular vector pairs $(u,v)$ 
    are characterized by $v$
    in the sense that the left singular vector $u$ 
    is given by $u = (A - V)v /\norm{(A - V)v}$.
    Starting in $(u^0, v^0)$,
    by Remark~\ref{rem:OpNorm2_detecting_null_maps} 
    we have $u^0 \not\in \ker((A - V)^*)$, 
    and $v^0 \not\in \ker(A - V)$.

    If the search directions $(w^{0}, x^{0})$ 
    were chosen such that $u^{1}$ is a left singular vector 
    and $v^{1}$ a right singular vector to a common singular value, 
    then we necessarily have 
    $(u^1, v^1) = \left(\tfrac{(A - V)v^1}{\norm{(A - V)v^1}}, v^1 \right)$.
    We also have $v^1 \not \in \ker(A - V)$,
    due to monotonicity, see Corollary~\ref{cor:monotonicity_abs_ak_in_OpNorm3}.
    Hence,
    for fixed $v^1$ this happens with probability zero, 
    since $w^0$ is uniformly in $\sphere^{m-2}$.
    Integrating over all possible $v^1$, 
    Fubini's Theorem shows that the whole event
    that $(u^1,v^1)$ is a pair of left and right singular vectors 
    to a common singular value
    has probability zero. 
    By induction, 
    we conclude that, almost surely, 
    for all $k \in \N$ the iterate $(u^k,v^k)$ 
    is not a pair of singular vectors to a common singular value
    and by Lemma~\ref{lem:update_step_as} we conclude that $a_k\neq 0$. 
    Now,
    all assertions above follow, 
    including Theorem~\ref{thm:convergence_B}.}
\end{remark}

\begin{remark}[Stopping criteria]
    \label{rem:stopping_criteria_alg_2}
    For a stopping criterion we use Corollary~\ref{cor:c_k_to_0}
    and Corollary~\ref{cor:ab+cd_and_b_to_0}
    to conclude that $|b_k| + |c_k| \to 0$ for $k \to \infty$ almost surely.
    Since $a_k, d_k$ are bounded,
    we also have $a_kb_k + c_kd_k \to 0$ for $k \to \infty$ (a.s.)
    and since Proposition~\ref{prop:factos_alomst_surely_not_zero} 
    claims that $a_kb_k + c_kd_k = 0$ happens only at pairs 
    of left and right singular vectors,
    we propose to use $|b_k| + |c_k| < \varepsilon$ as appropriate stopping criteria.
\end{remark}

\subsection{Convergence rates}
\label{sec:convergence_analysis}

In this subsection,
we aim to analyze the convergence rate of Algorithm~\ref{alg:OpNorm3}.

\begin{lemma}
    \label{lem:norm_proj_relation}
    Let $v \in \sphere^{d-1}$.
    For any $x \in \sphere^{d-1}$ holds 
    \begin{align*}
        \norm{(I_d - vv^*)x}^2 & = \norm{x - \scp{v}{x}v}^{2} = \norm{x - v}^2\Bigl(1 - \bigl(\tfrac{1}{2}\norm{x - v}\bigr)^2\Bigr)\\
        & = \tfrac{1}{4} \norm{x - v}^2 \norm{x + v}^2
        = 1 - \scp{x}{v}^2.
    \end{align*}
\end{lemma}
\begin{proof}
    All identities follow from straightforward manipulations.
\end{proof}

The next result shows that the angle between $u^{k}$ and $(A-V)v^{k}$ goes to zero at a sublinear rate.
\adj{For this, we chose under Assumption~\ref{ass:assump_2}
for some $1 > \delta \gg 0$ a constant $\kappa_0 > 0$ 
such that 
\begin{equation}
    \label{eq:kappa_0_delta}
    \PP(\abs{\scp{u^0}{(A - V) v^0}} > \kappa_0) > \delta.
\end{equation}
This is just an assumption on $A - V$,
i.e. on the location of the singular values.}

\begin{proposition}
    \label{prop:conv_rate_angle}
    \adj{Let $(u^k, v^k)$ the pair of random variables 
    generated by Algorithm~\ref{alg:OpNorm3}.
    Then, for any $\varepsilon > 0$ and any $n \in \NN$ holds 
    \begin{align*}
        \left.
        \begin{array}{c}
            \min_{0 \leq k \leq n}
            \PP(\norm{(I_m - u^k (u^k)^*)(A-V) v^k}^2 > \varepsilon) \\
            \min_{0 \leq k \leq n}
            \PP(\norm{(I_m - v^k (v^k)^*)(A-V)^{*} u^k}^2 > \varepsilon)
        \end{array}
        \right\}
        \leq 2\tfrac{\max\{m-1, d-1\}\norm{A - V}^2}{(n+1)\varepsilon},
    \end{align*}
    additionally, 
    for any $\varepsilon > 0$ under the choice from \eqref{eq:kappa_0_delta} 
    holds
    \[
        \min_{0 \leq k \leq n}
        \PP\Bigl( 1 - \scp{u^k}{\tfrac{(A-V) v^k}{(A-V) v^k}}^2
         + 1 - \scp{v^k}{\tfrac{(A-V)^* u^k}{(A-V)^* u^k}}^2 > \tfrac{\varepsilon}{\kappa_0} 
         \Bigl|\Bigr. \abs{a_0} > \kappa_0 \Bigr)
        \in \bigO(1/n)
    \]
    for and $\varepsilon > 0$ and $n \to \infty$.}
\end{proposition}
\begin{proof}
    See Appendix~\ref{proof:conv_rate_angle}.
\end{proof}

We can turn the above proposition 
into a convergence rate for the error in the eigenvector equation for $(A - V)^*(A - V)$.

\begin{theorem}
    \label{thm:conv_rate}
    \adj{Let $(u^k, v^k)$ the pair of random variables 
    generated by Algorithm~\ref{alg:OpNorm3}.
    Then, for any $\varepsilon > 0$ under the choice from \eqref{eq:kappa_0_delta} 
    holds for all $\varepsilon>0$
    \begin{align*}
         \min_{0 \leq k\leq n}
         \PP\Bigl(\norm{(A - V)^*(A - V) v^k - \lambda_k v^k}^2 > \tfrac{\varepsilon}{\kappa_0} 
         \Bigl|\Bigr. \abs{a_0} > \kappa_0\Bigr)
         \in \bigO(1/n)
         \quad \textnormal{as}\ n\to\infty,
    \end{align*}
    where $\lambda_k = \norm{(A - V)^* u^k} \cdot \norm{(A - V) v^k}$.}
\end{theorem}
\begin{proof}
    We use Lemma~\ref{lem:norm_proj_relation} to rewrite  
    \begin{align}
        \label{eq:diff_inner_prod_square}
        1 - \scp{v^k}{\tfrac{(A - V)^* u^k}{\norm{(A - V)^* u^k}}}^2 = \tfrac14\norm[\Big]{u^{k} - \tfrac{(A-V)v^{k}}{\norm{(A-V)v^{k}}}}^{2}\norm[\Big]{u^{k} + \tfrac{(A-V)v^{k}}{\norm{(A-V)v^{k}}}}^{2}.
    \end{align}
    By Theorem~\ref{thm:convergence_B}
    we already know that $\norm[\big]{u^{k} - \tfrac{(A-V)v^{k}}{\norm{(A-V)v^{k}}}}\to 0$ for $k \to \infty$, almost surely,
    which implies that $\norm[\big]{u^{k} + \tfrac{(A-V)v^{k}}{\norm{(A-V)v^{k}}}}\to 2$ almost surely.
    Similarly,
    we get that $\norm[\big]{v^{k} - \tfrac{(A-V)^{*}u^{k}}{\norm{(A-V)^{*}u^{k}}}}\to 0$ for $k \to \infty$ almost surely,
    which implies that $\norm[\big]{v^{k} + \tfrac{(A-V)^{*}u^{k}}{\norm{(A-V)^{*}u^{k}}}}\to 2$ almost surely.
    \adj{By construction with $a_k = \scp{u^k}{(A - V) v^k} \geq a_0 \geq 0$,
    we obtain the crucial bound 
    \[
        4 \geq \norm[\big]{u^{k} + \tfrac{(A-V)v^{k}}{\norm{(A-V)v^{k}}}}^2
        = 2\Bigl(1 + \scp{u^k}{\tfrac{(A-V)v^{k}}{\norm{(A-V)v^{k}}}}\Bigr)
        \geq 2\bigl(1 + \tfrac{a_0}{\norm{(A-V)v^{k}}}\bigr) 
        \geq 2
    \]
    and similar $2 \geq \norm[\big]{v^{k} + \tfrac{(A-V)^{*}u^{k}}{\norm{(A-V)^{*}u^{k}}}} \geq \sqrt{2}$.}
    Together with Proposition~\ref{prop:conv_rate_angle} this leads to 
    \begin{align}
        \label{eq:conv-rate-differences}
        \min_{0\leq k\leq n} 
        \PP\scalebox{0.9}{$\left(\norm[\Big]{u^{k} - \tfrac{(A-V)v^{k}}{\norm{(A-V)v^{k}}}}^{2} 
        \!\!\!+ \norm[\Big]{v^{k} - \tfrac{(A-V)^{*}u^{k}}{\norm{(A-V)^{*}u^{k}}}}^{2} \!\!\!> \tfrac{\varepsilon}{\kappa_0}
        \Bigl|\Bigr. \abs{a_0} > \kappa_0\right)$}
        \in \bigO(1/n)\quad\textnormal{for} n\to \infty.
    \end{align}
    Incorporating $\lambda_{k}$ as defined above,
    it holds that 
    \begin{align}
        \label{eq:inequality_to_prop}
        & \norm{(A-V)^{*}(A-V)v^{k} - \lambda_{k}v^{k}}^{2} \notag \\
        & \leq 2\Big(\norm[\big]{(A-V)^{*}(A-V)v^{k} - \norm{(A-V)v^{k}}(A-V)^{*}u^{k}}^{2} \notag \\
        & \qquad + \norm[\big]{\norm{(A-V)v^{k}}(A-V)^{*}u^{k} - \lambda_{k}v^{k}}^{2}\Big) \notag \\
        & \leq 2\Big(\norm{A-V}^{2}\norm{(A-V)v^{k}}^{2}\norm[\Big]{\tfrac{(A-V)v^{k}}{\norm{(A-V)v^{k}}} - u^{k}}^{2} \notag \\
        & \qquad  + \norm{(A-V)v^{k}}^{2}\norm{(A-V)^{*}u^{k}}^{2} \norm[\Big]{\tfrac{(A-V)^{*}u^{k}}{\norm{(A-V)^{*}u^{k}}} - v^{k}}^2\Big) \notag \\
        & \leq 2\norm{A-V}^{4} \left[\norm[\Big]{\tfrac{(A-V)v^{k}}{\norm{(A-V)v^{k}}} - u^{k}}^{2} + \norm[\Big]{\tfrac{(A-V)^{*}u^{k}}{\norm{(A-V)^{*}u^{k}}} - v^{k}}^{2} \right].
    \end{align}
    Utilizing~\eqref{eq:conv-rate-differences} in \eqref{eq:inequality_to_prop}
    shows the claim.
\end{proof}

The quantity $\lambda_{k}$ in the previous theorem is not computable in practice.
However,
the next results shows the same convergence rate
for an eigenvector equation,
where the eigenvalue is approximated by $a_k^2$ .

\begin{corollary}
    \label{cor:conv_rate_a_k}
    \adj{Let $(u^k, v^k)$ the pair of random variables 
    and $a_k$ generated by Algorithm~\ref{alg:OpNorm3}.
    Then, for any $\varepsilon > 0$ under the choice from \eqref{eq:kappa_0_delta} 
    holds for all $\varepsilon>0$
    \begin{align*}
        \min_{0 \leq k \leq n}
        \PP\bigl( \norm{(A - V)^*(A - V) v^k - a_k^2 v^k}^2 > \tfrac{\varepsilon}{\kappa_0}
        \;\bigl|\bigr.\; \abs{a_0} > \kappa_0\bigr)
        \in \bigO(1/n)
        \quad \textnormal{for}\ n\to\infty.
    \end{align*}}
  \end{corollary}
\begin{proof}
See Appendix~\ref{proof:conv_rate_a_k}.
\end{proof}

Since we have $\scp{u^k}{A v^k} \leq \norm{A v^k}$
which is the approximation of the maximal singular value in the $k$th iteration of \cite[Alg.~1]{BLSW24},
we also obtain the convergence rate of $\bigO(n^{-1/2})$ in \cite{BLSW24}, 
i.e. we have for all $\varepsilon>0$
\[
  \min_{0 \leq k \leq n}
  \adj{\PP(\norm{A^*A v^k - \norm{A v^k}^2 v^k} > \varepsilon)
  \in \bigO(n^{-1/2}),\quad\textnormal{for}\ n\to \infty,}
\]
due to the similar architectures of the methods
and $a_0 = \norm{A v^0}^2 > 0$ almost surely,
if $v^0 \in \sphere^{d-1}$ is uniformly sampled.

\section{Examples and numerical experiments}
\label{sec:numerical-experiments}

Before we state numerical experiments we give an illustrative example:
Consider the matrices    
\begin{align*}
    A = \left[\begin{smallmatrix}
        1 & 0 \\ 0 & 0
    \end{smallmatrix}\right]
\qquad \text{and}\qquad A = 
    \left[\begin{smallmatrix}
        1 & 0 \\ 0 & 1 \\ 0 & 0
    \end{smallmatrix}\right]
\end{align*}
and $V$ is set to be the zero operator of respective size.
For the first matrix,
Algorithm~\ref{alg:OpNorm3} arrives at the optimum in exactly one iteration.
Notably,
\cite[Alg.~1]{BLSW24} also converges in one iteration,
since $A^{*}A = A$.
For the second matrix,
Algorithm~\ref{alg:OpNorm3} converges in one iteration, again,
and \cite[Alg.~1]{BLSW24} would detect an orthogonal operator,
since $A^*A = I_2$.
Differently, in the case we would consider $A^* \in \RR^{2\times3}$ instead of $A$,
only the behavior of \cite[Alg.~1]{BLSW24} would change,
since we now have $AA^* = \textup{diag}(1,1,0)$.

\subsection{Performance analysis of the proposed algorithm}
    
For a first numerical experiment,
we \adj{study} the convergence speed of Algo\-rithm~\ref{alg:OpNorm3}.
To this end,
we apply Algo\-rithm~\ref{alg:OpNorm3} to 50 randomly generated Gaussian matrices 
of size $m \times d$ for different values $m$ and $d$
and plot the relative error, i.e. 
\begin{equation*}
    \adj{\tfrac{\norm{A - V} - a_k}{\norm{A - V}}
    = \tfrac{\norm{A - V} - \scp{u^k}{(A - V) v^k}}{\norm{A - V}}
    = \tfrac{\norm{A - V} - \scp{u^k}{A v^k} + \scp{V^* u^k}{v^k}}{\norm{A - V}}}
\end{equation*}
for all iterations, see Figure~\ref{fig:experiment1}.
For the higher dimensional cases, 
we observe that dimension of the output space has a siginificatn impact 
on the numbers of necessary iterations for an appropriate approximation result
for the adjoint mismatch.%
\footnote{The code for this and the other experiments can be found at \url{https://github.com/JJEWBresch/MatrixNormWithoutAdjoint}.}

\begin{figure}[t]
  \resizebox{\textwidth}{!}{
  \begin{threeparttable}
  \begin{tabular}{c c c c}
    & $10\times50$ & $50\times50$ & $100\times50$ \\
    \rotatebox{90}{\hspace{0.2cm}$\adj{\abs{b_k} + \abs{c_k}}$\tnote{I} \hspace{0.55cm} $\tfrac{\norm{A - V} - \abs{\scp{u^k}{(A - V) v^k}}}{\norm{A - V}}$}&
    \includegraphics[width=0.38\textwidth, clip=true, trim=10pt 40pt 30pt 60pt]{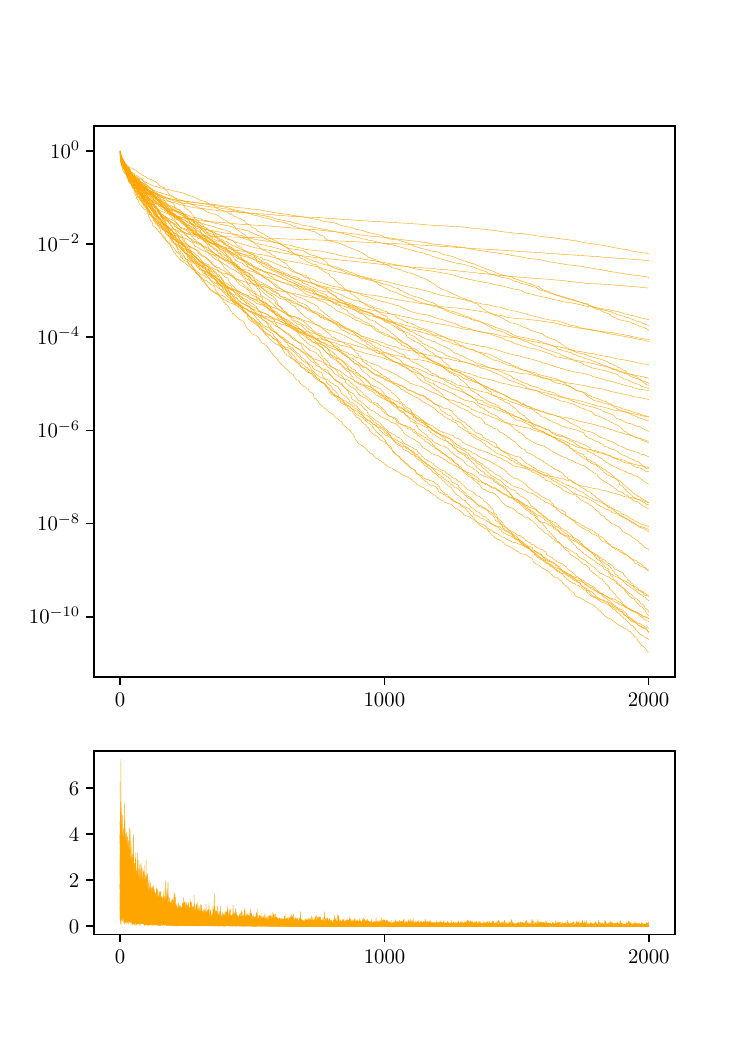}&
    \includegraphics[width=0.38\textwidth, clip=true, trim=10pt 40pt 30pt 60pt]{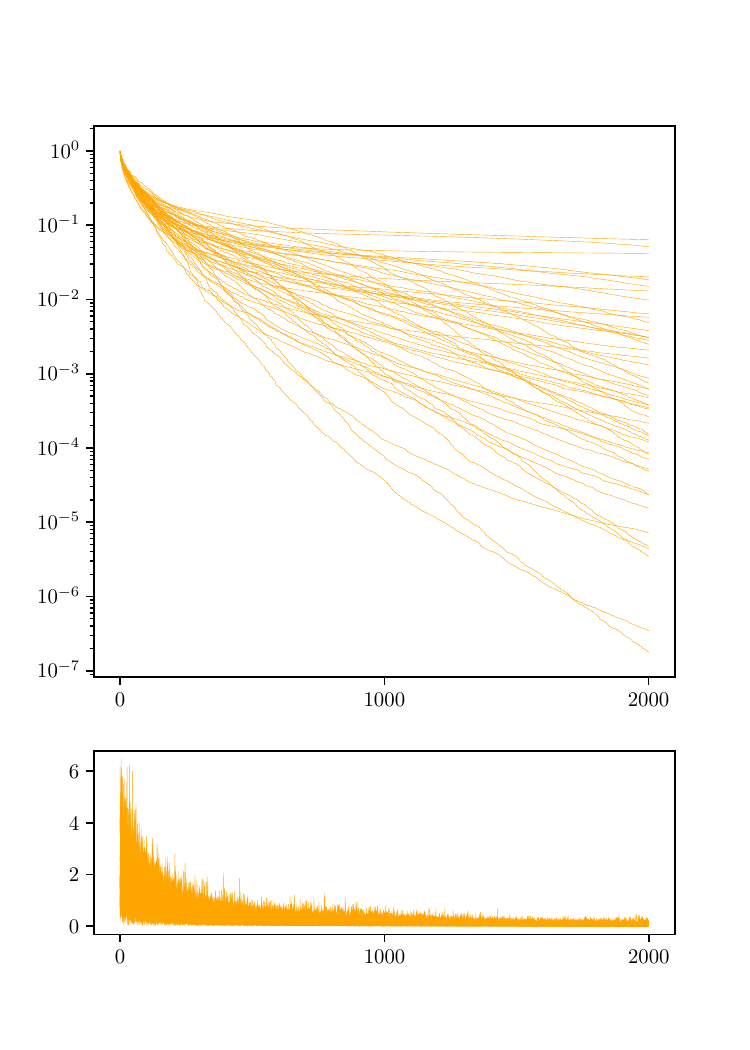}&
    \includegraphics[width=0.38\textwidth, clip=true, trim=10pt 40pt 30pt 60pt]{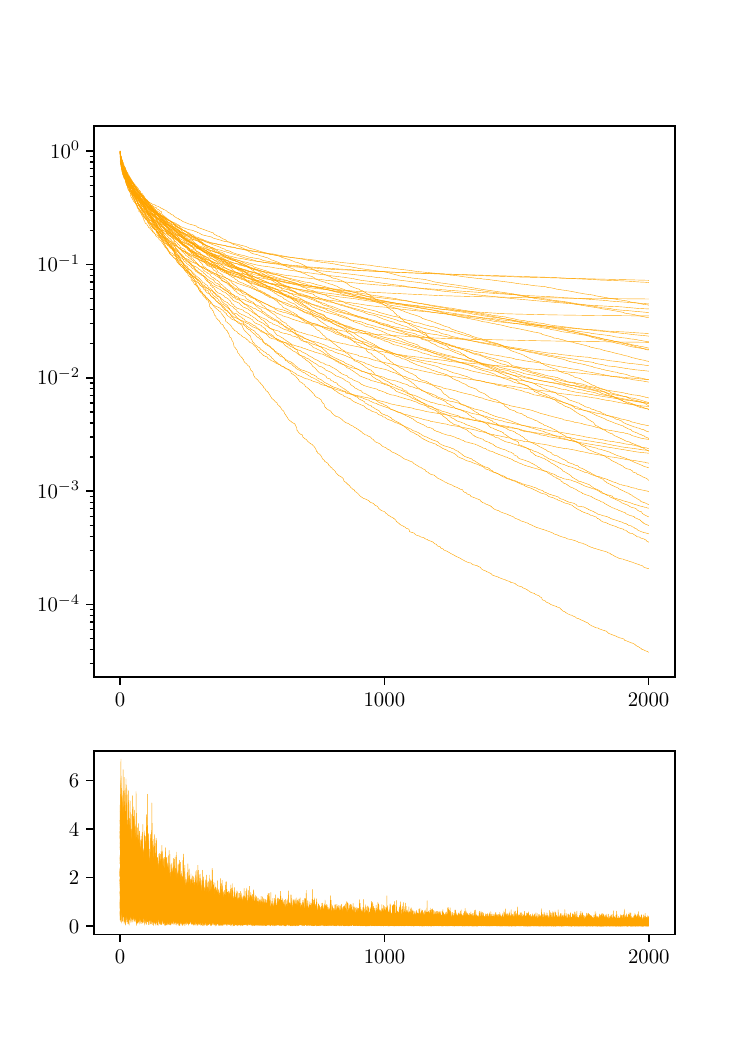}\\
    & $100\times500$ & $500\times500$ & $1000\times500$ \\
    \rotatebox{90}{\hspace{0.2cm}$\adj{\abs{b_k} + \abs{c_k}}$\tnote{I} \hspace{0.55cm} $\tfrac{\norm{A - V} - \abs{\scp{u^k}{(A - V) v^k}}}{\norm{A - V}}$}&
    \includegraphics[width=0.38\textwidth, clip=true, trim=10pt 40pt 30pt 60pt]{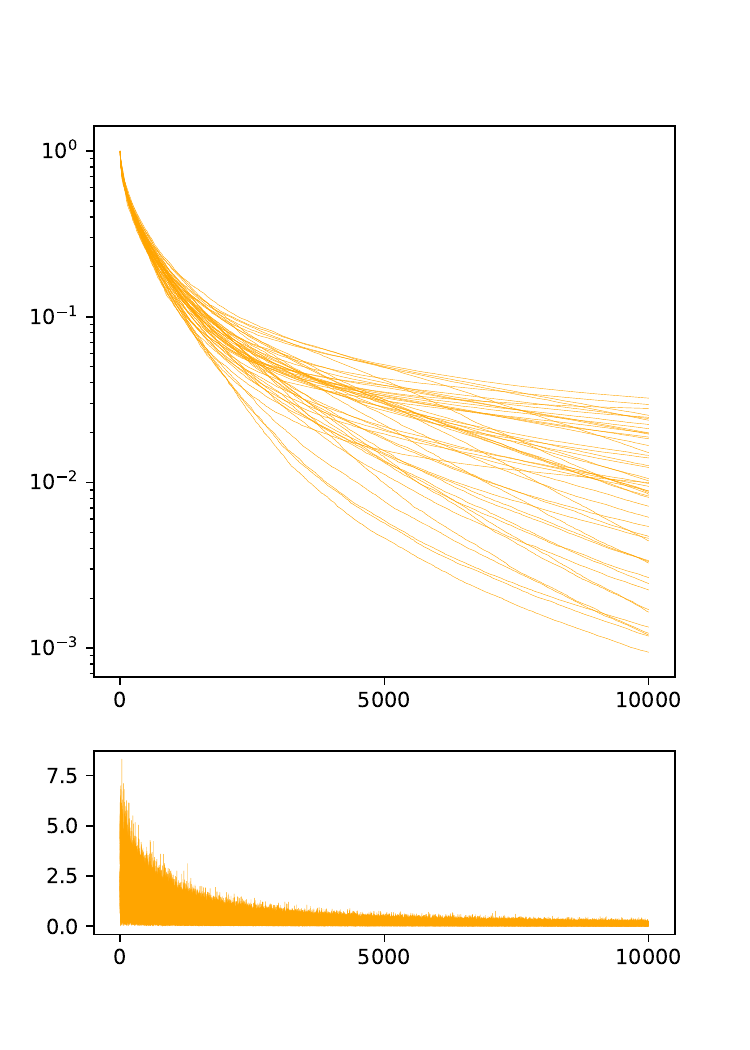}&
    \includegraphics[width=0.38\textwidth, clip=true, trim=10pt 40pt 30pt 60pt]{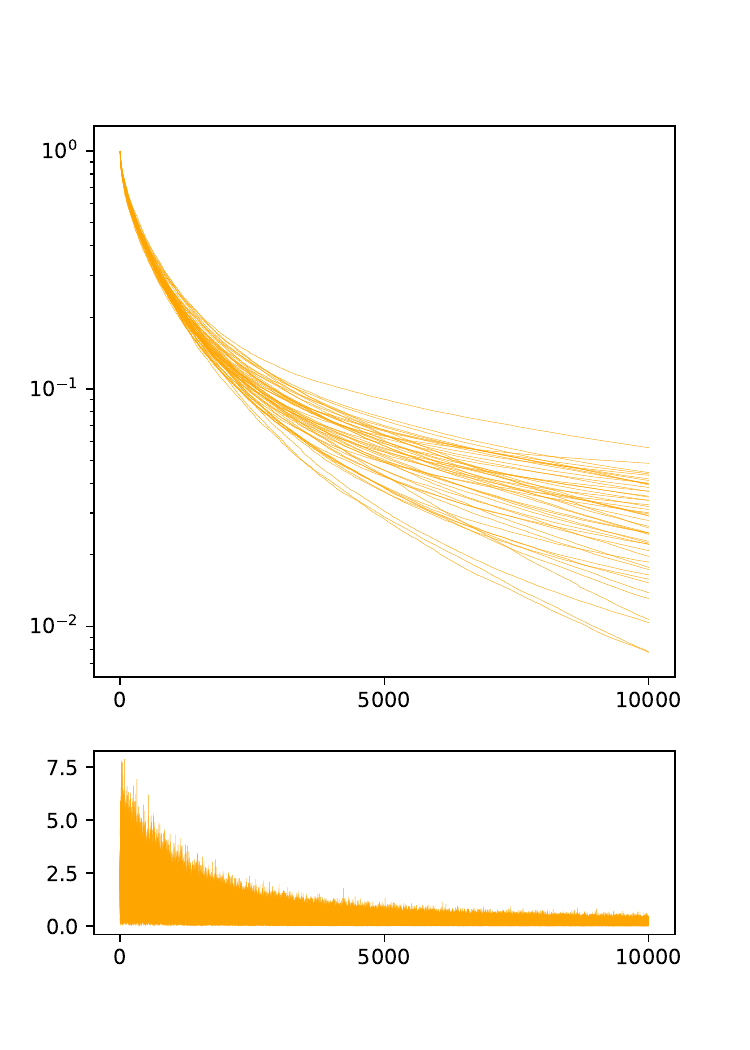}&
    \includegraphics[width=0.38\textwidth, clip=true, trim=10pt 40pt 30pt 60pt]{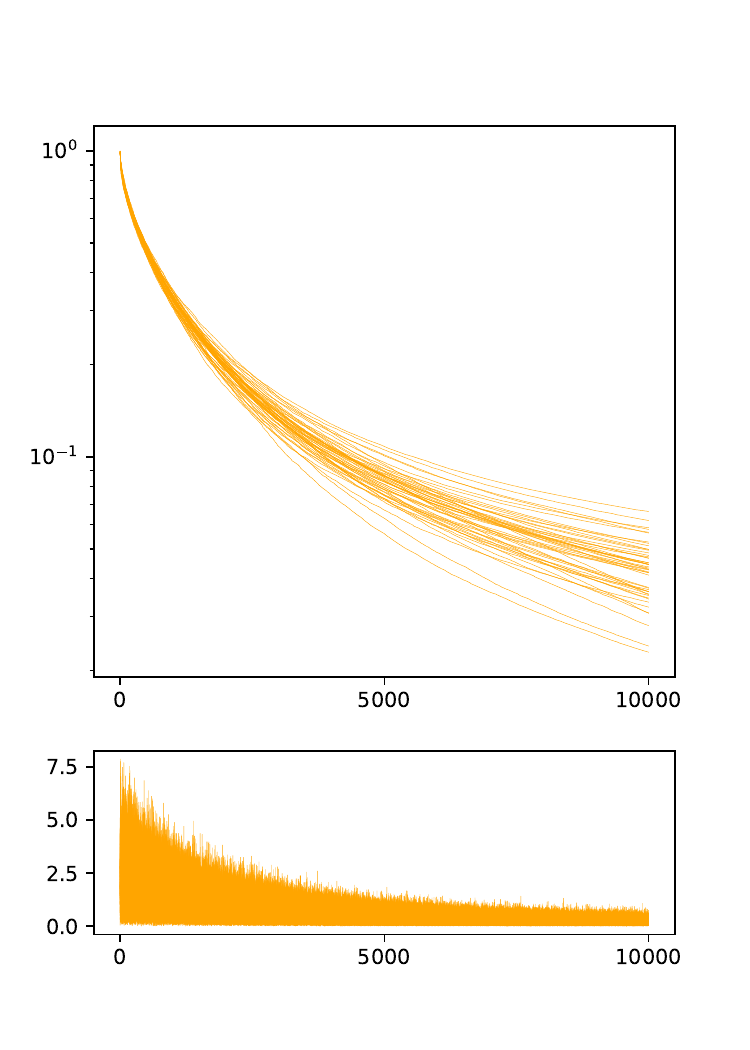}
  \end{tabular}
  \begin{tablenotes}
    \item[I] The absolute value of the sum of two of the four parameters \adj{from \eqref{eq:abcd_parameter}} generated from Algorithm~\ref{alg:OpNorm3}.
  \end{tablenotes}
  \end{threeparttable}
  }
  \caption{Results for 50 runs of Algorithm~\ref{alg:OpNorm3}
  for Gaussian matrices of different sizes $m\times d$.}
  \label{fig:experiment1}
\end{figure}

\subsection{Comparison with \cite{BLSW24} for computing $\norm{A}$}

As a second experiment,
we consider the algorithm from \cite[Alg.~1]{BLSW24} which has been proposed for computation of the norm $\|A\|$ without using the adjoint of $A$,
in comparison with Algorithm~\ref{alg:OpNorm3}.
Therefore,
we run both algorithms on 50 randomly generated Gaussian matrices 
$A \in \RR^{m\times d}$, 
and set $V^* = 0 \in \RR^{d \times m}$ for Algorithm~\ref{alg:OpNorm3}.
The relative errors are displayed in Figure~\ref{fig:experiment2}.
Additionally,
we give the value of $a_k$,
generated by \adj{\cite[Alg.~1]{BLSW24}},
and the absolute value of $b_k$ plus $c_k$,
generated by Algorithm~\ref{alg:OpNorm3}.
Remarkably, the convergence rate of \cite[Alg.~1]{BLSW24}
is faster than the one of Algorithm~\ref{alg:OpNorm3}.
We observe that the superiority of \cite[Alg.~1]{BLSW24} 
is more visible with increasing dimension of the output space.

\begin{figure*}[htp]
  \resizebox{\textwidth}{!}{
  \begin{threeparttable}
  \begin{tabular}{c c c c}
    & $10\times50$ & $50\times50$ & $100\times50$ \\
    \rotatebox{90}{\hspace{0.2cm}$\adj{\abs{b_k} + \abs{c_k}}$\tnote{II} \hspace{0.8cm} $a_k$\tnote{I} \hspace{1.25cm} $\tfrac{\norm{A - V} - \abs{\scp{u^k}{(A - V) v^k}}}{\norm{A - V}}$}&
    \includegraphics[width=0.38\textwidth, clip=true, trim=10pt 60pt 30pt 80pt]{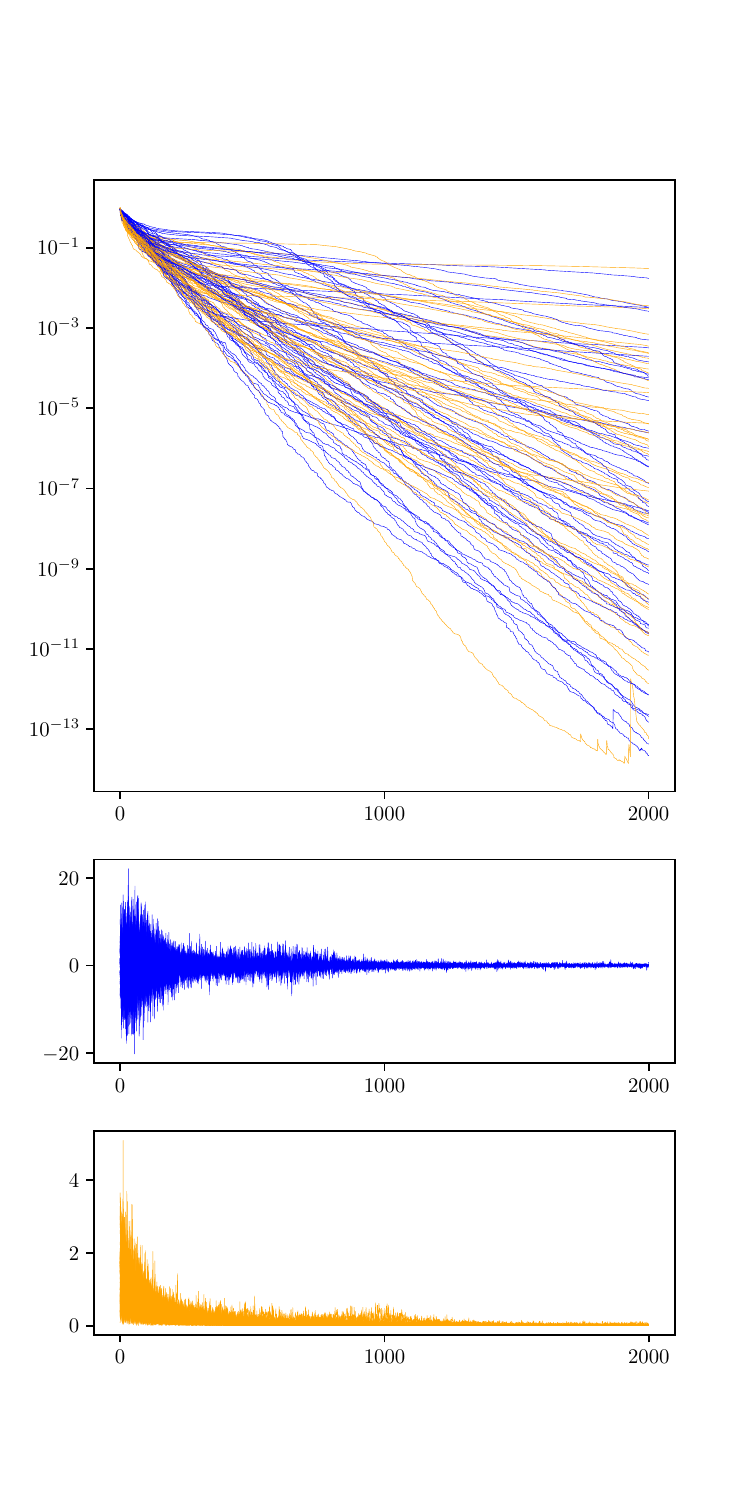}&
    \includegraphics[width=0.38\textwidth, clip=true, trim=10pt 60pt 30pt 80pt]{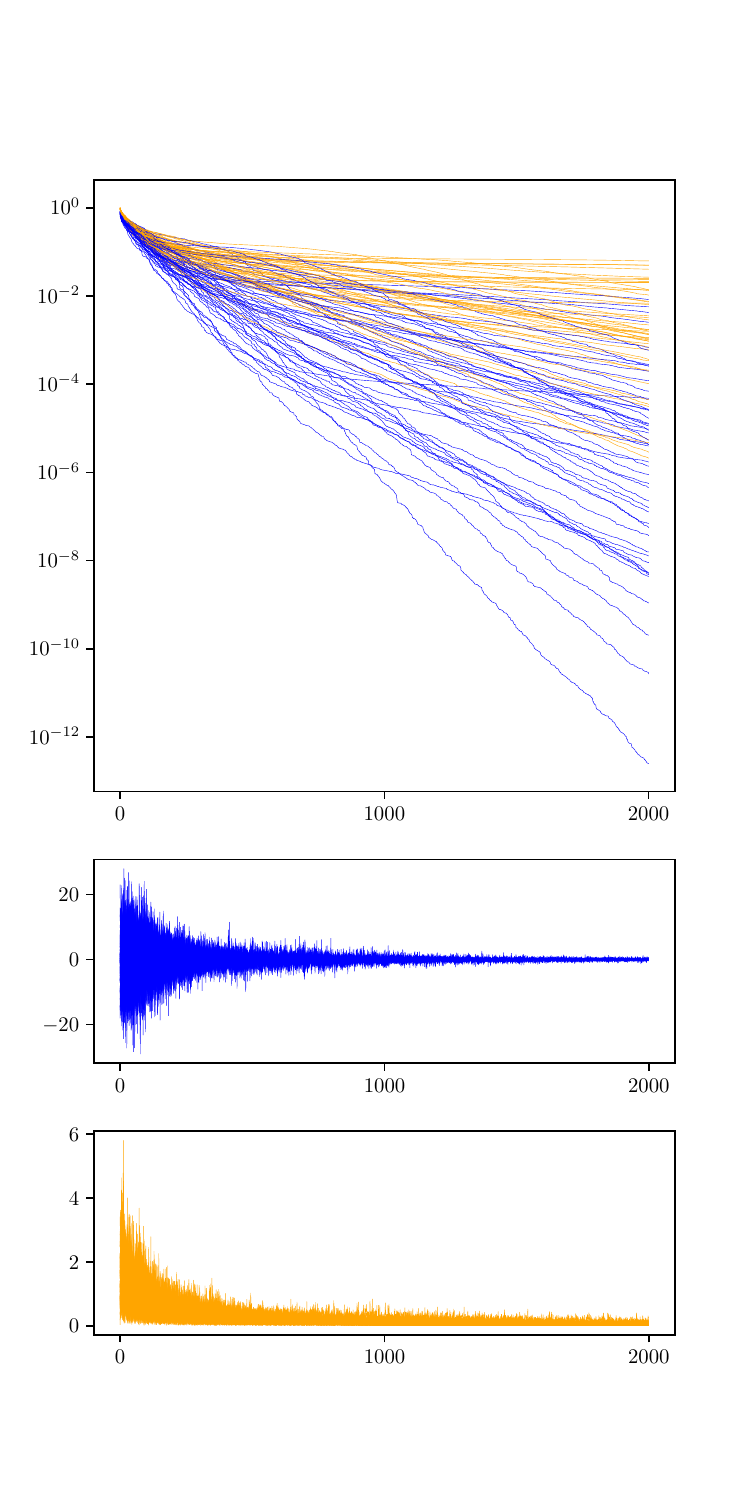}&
    \includegraphics[width=0.38\textwidth, clip=true, trim=10pt 60pt 30pt 80pt]{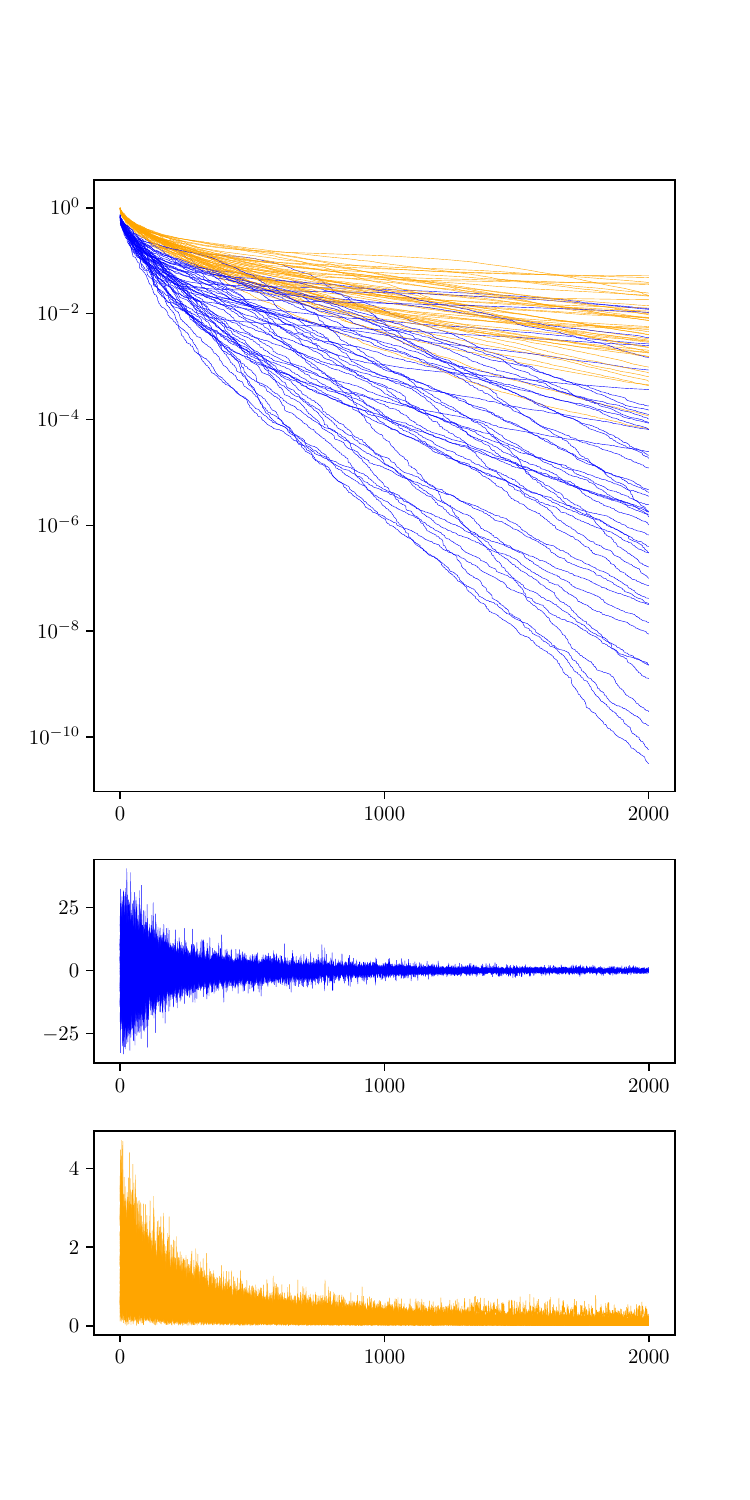}\\
    & $100\times500$ & $500\times500$ & $1000\times500$ \\
    \rotatebox{90}{\hspace{0.2cm}$\adj{\abs{b_k} + \abs{c_k}}$\tnote{II} \hspace{0.8cm} $a_k$\tnote{I} \hspace{1.25cm} $\tfrac{\norm{A - V} - \abs{\scp{u^k}{(A - V) v^k}}}{\norm{A - V}}$}&
    \includegraphics[width=0.38\textwidth, clip=true, trim=10pt 60pt 30pt 80pt]{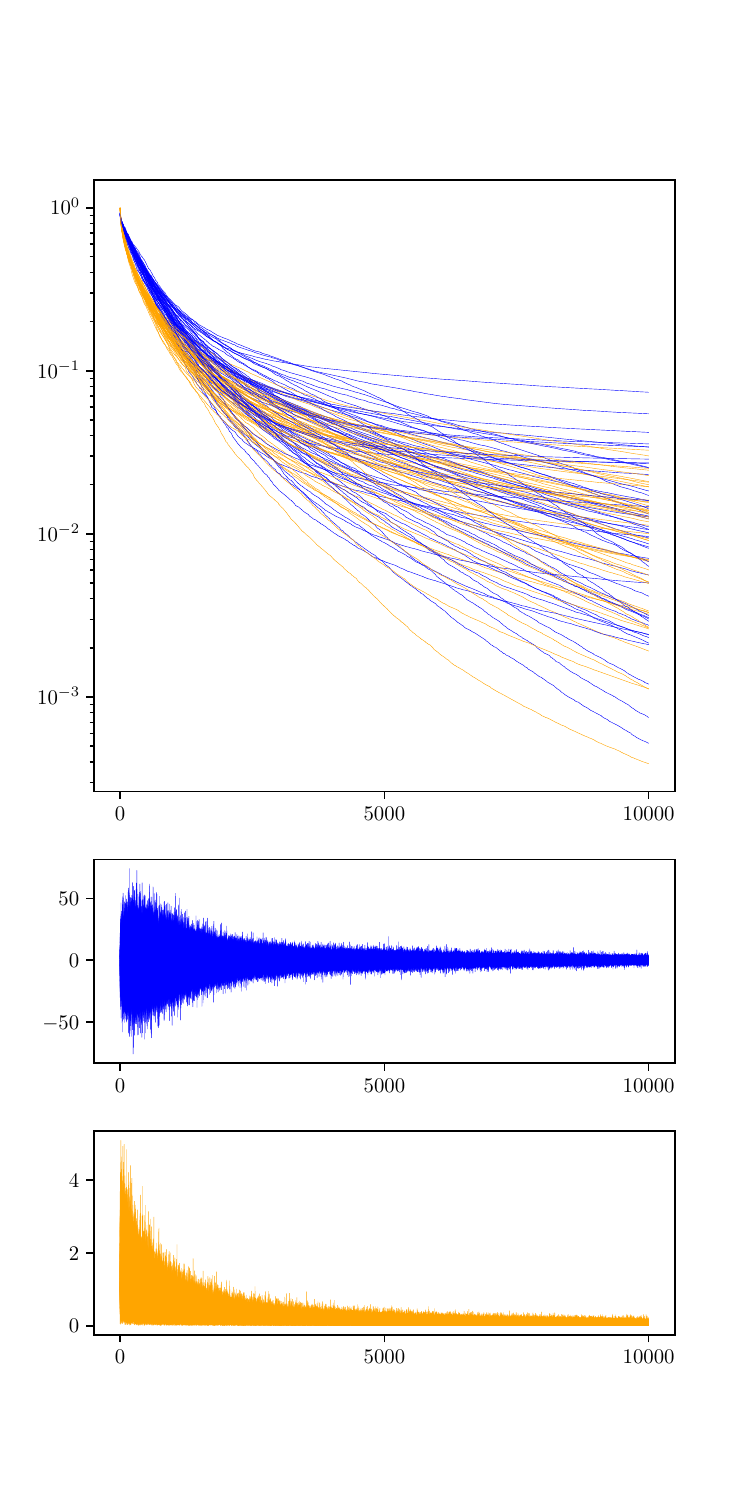}&
    \includegraphics[width=0.38\textwidth, clip=true, trim=10pt 60pt 30pt 80pt]{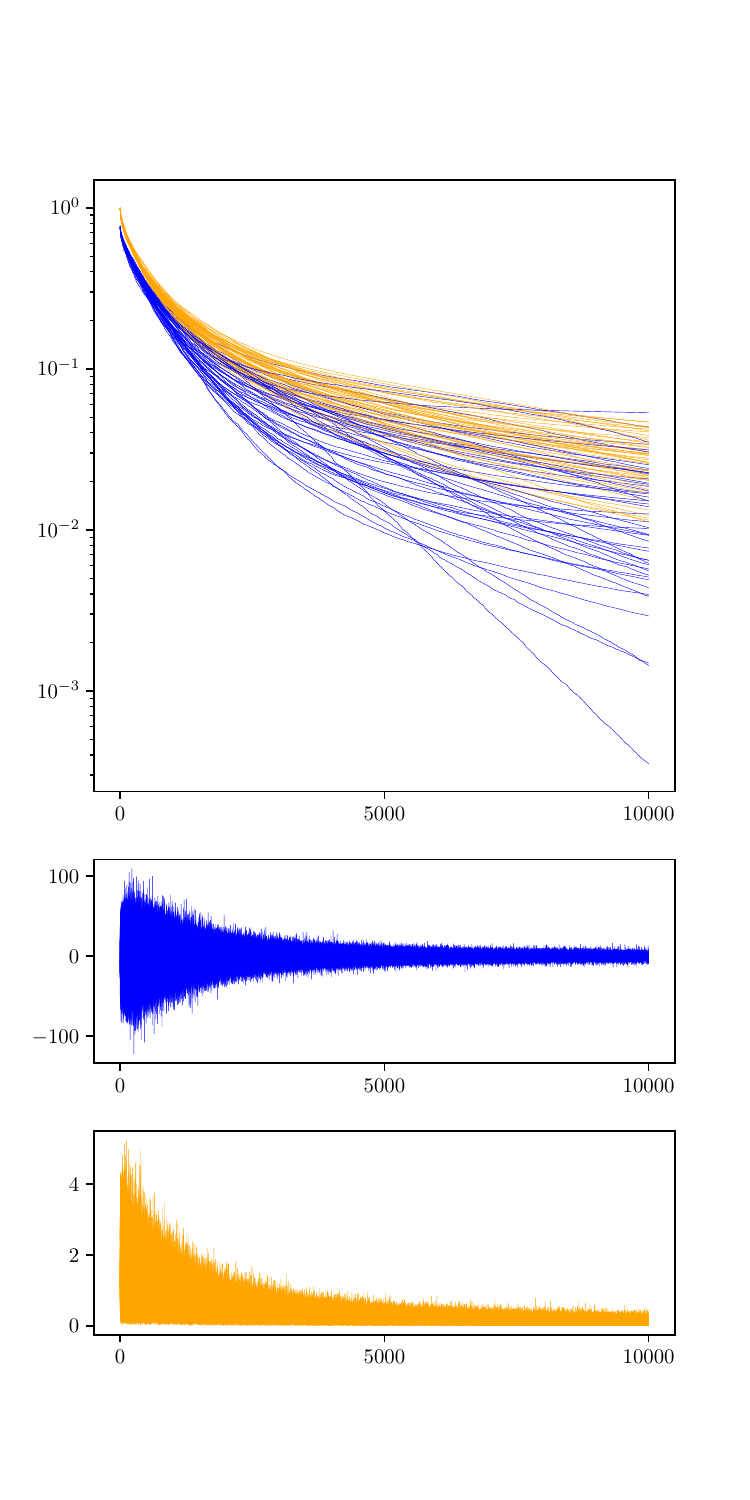}&
    \includegraphics[width=0.38\textwidth, clip=true, trim=10pt 60pt 30pt 80pt]{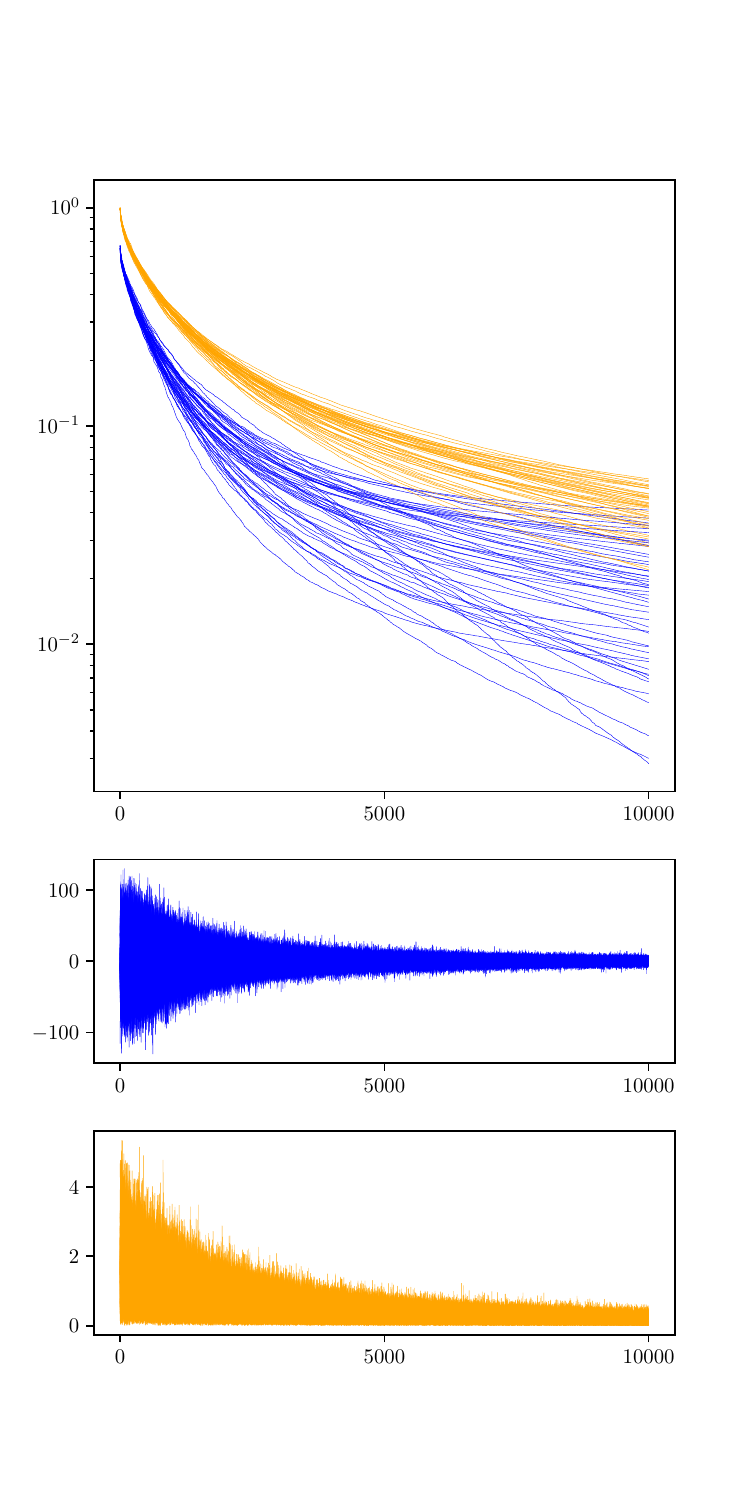}
  \end{tabular}
  \begin{tablenotes}
    \item[I] The parameter \adj{$a_{k}$ from \cite[Eq.~7]{BLSW24}} generated by \cite[Alg.~1]{BLSW24}.
    \item[II] The absolute value of the sum of two of the four parameters \adj{from \eqref{eq:abcd_parameter}} generated by Algorithm~\ref{alg:OpNorm3}.
  \end{tablenotes}
  \end{threeparttable}}
  \caption{Results for 50 runs of \cite[Alg.~1]{BLSW24} (blue)
  and Algorithm~\ref{alg:OpNorm3} (orange) 
  for Gaussian matrices $A$, where $V = 0$
  and $u^k = A v^k / \norm{A v^k}$ for \cite[Alg.~1]{BLSW24},
  of different sizes.}
  \label{fig:experiment2}
\end{figure*}

\subsection{Checking the adjoint mismatch of Radon implementation in Astra}

In a third experiment, 
we aim to demonstrate that the for\-ward and back\-pro\-jec\-tions implemented 
in the Astra toolbox~\cite{van2016fast} 
are indeed adjoint to each other. 
To this end, we consider different projections of discrete \(400 \times 400\) 
images with 40 uniformly distributed angles 
and 400 detector pixels per angle, 
resulting in mappings from \(\mathbb{R}^{400 \times 400}\) 
to \(\mathbb{R}^{40 \times 400}\). 
In our example, we restrict ourselves to parallel beam geometries. 
We denote by \(\mathcal{R}_1\) the projection given by the line model, 
by \(\mathcal{R}_2\) the projection using the ray model
with rays of the width of a detector pixel, 
and by \(\mathcal{R}_3\) the projection using the \emph{Joseph method}.
For \(i = 1, 2, 3\), we describe the corresponding backprojections, 
provided by the command \texttt{create\_backprojection}, 
as \(\mathcal{B}_i\). If \(\mathcal{B}_i^* = \mathcal{R}_i\), 
then Algorithm~\ref{alg:OpNorm3} should always yield a value of zero,
or a very small value due to numerical errors. 
Accordingly,
we apply the algorithm to the pairs \(A = \mathcal{R}_i\) 
and \(V^* = \mathcal{B}_i\) and terminate the iteration after 1,000 repetitions. 
The results are as follows:
\[
\|\mathcal{R}_1 - \mathcal{B}_1^*\| \approx 5.7064 \cdot 10^{-9}, \;
\|\mathcal{R}_2 - \mathcal{B}_2^*\| \approx 1.1471 \cdot 10^{-8}, \;
\|\mathcal{R}_3 - \mathcal{B}_3^*\| \approx 5.4536 \cdot 10^{-9}.
\]
This demonstrates that the forward and back projections, 
are truly adjoint to each other.
The Radon transform is implemented in MATLAB as \texttt{radon} and the back projection can be computed by \texttt{iradon} with the option \texttt{filter="none"}.
As already observed in~\cite{BLSW24}, these operators are not adjoint to each other, and with the method from this paper we can calculate their norm difference.
In the same geometry as for Astra (images of size $400\times 400$, $40$ angles and $400$ pixels on the detector) we get a relative norm difference 
between the true adjoint of \texttt{radon} and \texttt{iradon} with option \texttt{filter="none"}
of at least \(0.1\) (after just $1000$ iterations).

\subsection{Numerical evidence of the convergence rates}

In a fourth experiment,
we aim to verify the convergence speed for the singular vector and value equation 
from Corollary~\ref{cor:conv_rate_a_k}.
To this end,
we run Algorithm~\ref{alg:OpNorm3} 50 times 
for randomly sampled Gaussian matrices $A \in \RR^{50 \times 100}$
and $V^* \in \RR^{100 \times 50}$
and $A \in \RR^{100 \times 50}$
and $V^* \in \RR^{50 \times 100}$.
The results are given in Figure~\ref{fig:experiment4}.
Moreover,
we observe nearly exponential convergence.
We conclude that the non-exponential convergence rate 
from Corollary~\ref{cor:conv_rate_a_k} 
is quite conservative.

\begin{figure}
    \resizebox{\textwidth}{!}{
    \begin{threeparttable}
    \begin{tabular}{c c c}
        & $50\times100$ & $100\times50$ \\
    \rotatebox{90}{\hspace{6pt}$\displaystyle\min_{0 \leq k \leq n} \tfrac{\norm{(A - V)^*(A - V) v^k - \abs{a_k}^2 v^k}}{\norm{A - V}}$} 
    & \includegraphics[width=0.55\textwidth, clip=true, trim=30pt 20pt 40pt 40pt]{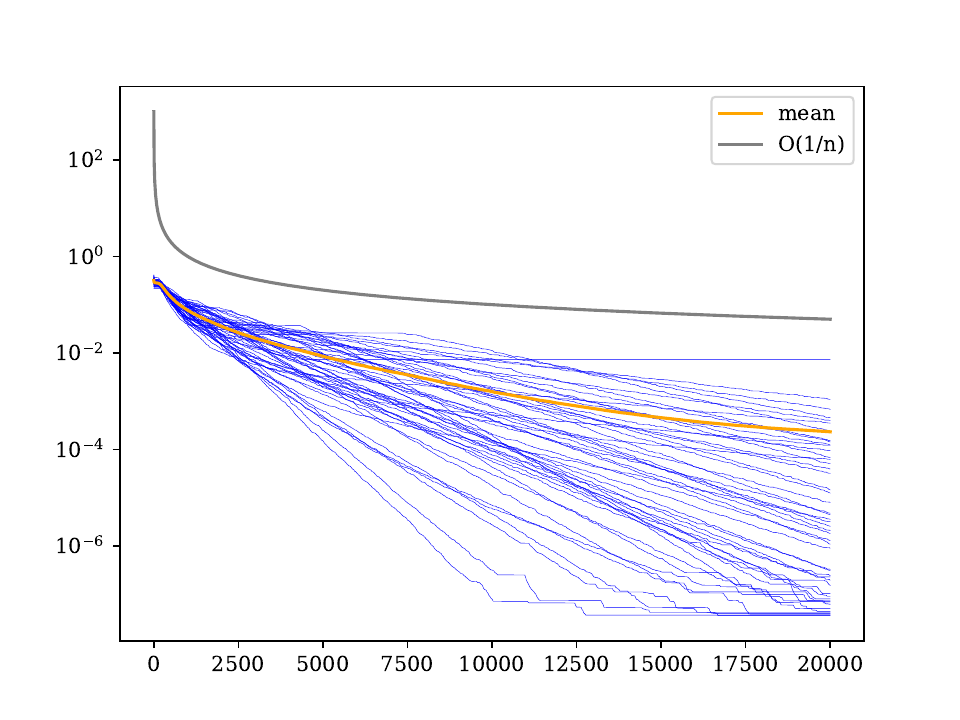}
    & \includegraphics[width=0.55\textwidth, clip=true, trim=30pt 20pt 40pt 40pt]{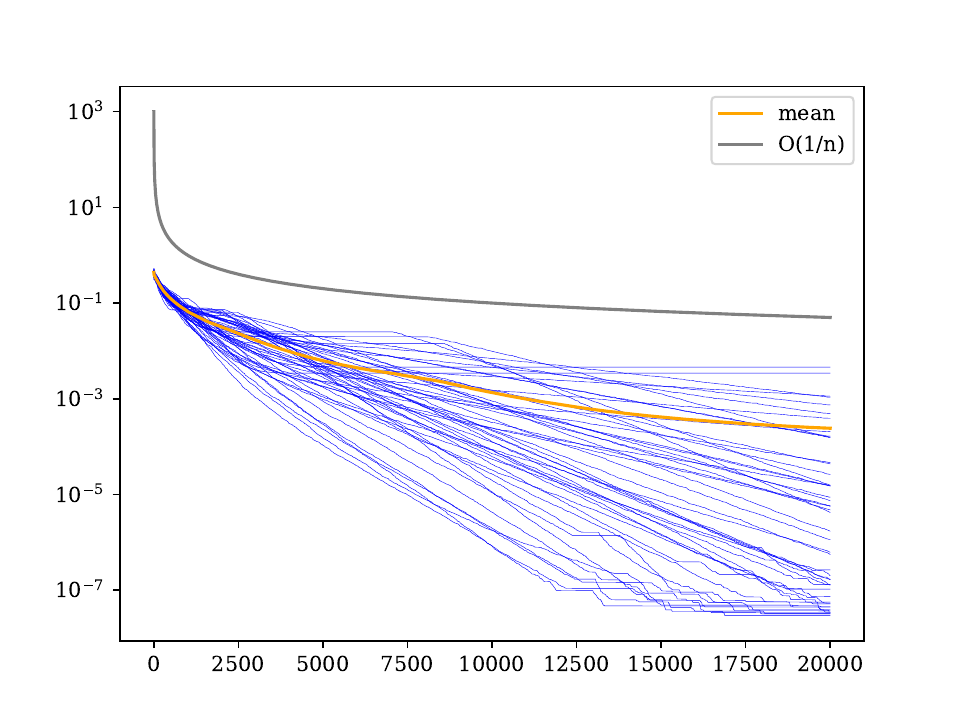}
    \end{tabular}
    \end{threeparttable}}
    \caption{Results for 50 runs of Algorithm~\ref{alg:OpNorm3}
    for Gaussian matrices of different sizes.
    The minimal error within the $n$-th iteration 
    for the corresponding eigenvalue equation 
    is given.}
    \label{fig:experiment4}
\end{figure}

\section{Conclusion and outlook}
\label{sec:conclusion}

The algorithms we developed in this paper allow to compute norm differences of operators based only on evaluations of one of the operators and evaluation of the adjoint of the other operator.
As such, the methods can be applied even if only black-box implementations are available and if the computing environment does not allow to assemble large parts of the matrix representations.
Our methods also produce pairs of right and left singular vectors (for the largest singular value) of the difference of the operators and can be extended to produce more singular vectors for the next few largest singular values. By slight changes our methods could also be modified to calculate the smallest singular value (and respective singular vectors) by focusing on minimization in the line-search instead of maximization. One particular application of our algorithms is to check the adjointness of projector and back projector pairs in computed tomography but also to compute the mismatch between them to make use of the reconstruction methods developed in~\cite{Chouzenout2021pgm-adjoint,Savanier2021ProximalGA,chouzenoux2023convergence,lorenz2023chambolle,naldi2025influence}.

\appendix

\section{Proofs}
\label{sec:appendix-A}

\subsection{Proof of Lemma~\ref{lem:normdifference-dimension-directions}}
\label{sec:proof-normdifference-dimension-directions}

From \eqref{eqn:E_sigma_special}
and since \(\dim(\tilde E_\sigma) = r < \max\{d,m\}\), 
we find a matrix \( C \in \RR^{(d + m - r) \times (d + m)} \) 
\adj{with orthonormal basis vectors
such that  $\tilde E_\sigma$ can be expressed by \( \tilde E_\sigma = \{ p \mid C p = 0\}\).
Now,
if it was true that $(u,v)^* + (w,x)^* \in \tilde E_\sigma$,
then it holds 
\begin{align*}
    C (u + w, v + x)^* = 0
    \quad \Leftrightarrow \quad 
    C (w,x)^* = - C(u,v)^*.
\end{align*}
This system of equations in $(w,x)^*$
might have no solution,
such that the set of directions fulfilling $(u+w, v+x)^* \in \tilde E_\sigma$
is empty.
Otherwise,
if $(\tilde w, \tilde x)^*$ is a solution of the latter sytsem of equations,
then the set of all such directions is given by 
\[
    (\tilde w, \tilde x)^* + \ker\left(\left[\begin{smallmatrix}
        C \\
        (u^*,0^*) \\
        (0^*,v^*)
    \end{smallmatrix}\right]\right).
\]
The dimension 
is at most $r$
and of dimension $r-1$ if either $(u,0)^* \in \langle C^*\rangle$
or $(0,v)^* \in \langle C^*\rangle$
and of dimension $r-2$ if both of the latter conditions are not fulfilled and might additionally be empty.
This finishes the proof.}
\qed


\subsection{Proof of Corollary~\ref{cor:normdistance-2-no-sing-vectors-during-iteration}}
\label{proof:normdistance-no-sing-vectors-during-iteration}
Let \(\sigma\) be a singular value of \(A-V\) with multiplicity \(r\), 
\adj{and let \(\tilde E_\sigma\) from \eqref{eqn:E_sigma_special} be the corresponding singular vector space.}

If \((u^k + \adj{\xi} w^k, v^k + \adj{\tau} x^k)^*\) lies in \(\tilde E_\sigma\)
\adj{for some stepsizes $\xi ,\tau \in \RR$}, 
then, by Lemma~\ref{lem:normdifference-dimension-directions}, 
the vector \(\hat{p} = (\adj{\xi} w^k, \adj{\tau} x^k)^*\) belongs to an affine subspace \(L\) 
of dimension at most \(r\). 
Thus, 
the set of directions 
\[
    (w^k, x^k) \in \left(\{u^k\}^\perp \cap \sphere^{m-1}\right) \times \left(\{v^k\}^\perp \cap \sphere^{d-1}\right)
\]
that could result in \((u^{k+1}, v^{k+1})^* \in \tilde E_\sigma\) 
is a subset of 
\(\left\{\adj{\sqrt{2}} x/\norm{x} \mid x \in L \right\}\). 
Since \(v^k\) and \(x^k\) as well as \(u^k\) and \(w^k\) 
are not parallel, 
this subset has measure zero with respect to the Haar measure 
on \(\left(\{u^k\}^\perp \cap \sphere^{m-1}\right) \times \left(\{v^k\}^\perp \cap \sphere^{d-1}\right)\), 
according to which \((w^k, x^k)\) is randomly chosen, cf.~Remark~\ref{rem:distribution-x-w-const}, 
if $L$ has dimension less \(\max\{m, d\} - 1\). 
By Lemma~\ref{lem:normdifference-dimension-directions}, 
this is ensured by \(r < \max\{m, d\} - 1\), 
which holds by \adj{Assumption~\ref{ass:assump_2}}. 

Since this applies to each of the finitely many spaces of singular vectors 
to the different (finitely many) singular values, 
the claim follows.
\qed

\subsection{Proof of Lemma~\ref{lem:normdifference-B_Lemma}:}
\label{proof:normdifference-B_Lemma}
\begin{enumerate}[i)]
\item
    This holds because $\scp{u^k}{(A-V)v^k}$ is constructed to increase with $k \in \N$,
    almost surely.
\item
    \adj{According to Proposition~\ref{prop:acc_point_opt_pair_two_stepsizes}, 
    it holds that there exists a subsequence $(u^{k_j}, v^{k_j})$
    and a random variable $\sigma$ with values in the set if singular values of $A - V$
    such that $\textrm{dist}(\tilde E_\sigma, (u^{k_j}, v^{k_j})^*) \to 0$ for $j \to \infty$
    almost surely for some singular value $\sigma$ of $A - V$.
    By (i) and assumption, it is $\sigma = \sigma_1$ almost surely.
    Furthermore,
    the sequence $(\scp{u^k}{(A-V)v^k})_{k \in \NN}$ 
    converges almost surely to the corresponding singular value
    by the monotonicity of $\scp{u^k}{(A-V)v^k}$
    (cf.~Corollary~\ref{cor:monotonicity_abs_ak_in_OpNorm3})
    and assumption, i.e. 
    \begin{align*}
        \scp{u^{k}}{(A-V)v^{k}}
        &> 
        \scp{u^{k_0}}{(A-V)v^{k_0}}
        > \sigma_{2}(A-V)
        \quad \forall k \geq k_0
        \quad \textrm{a.s.}
    \end{align*}
    Therewith the convergence of the objective to the largest singular value follows, 
    almost surely.
    
    Moreover,
    we can decompose $(u^k, v^k)$ into components of the singular vector spaces,
    i.e. eigen spaces of $A - V$,
    with, say $s$, distinct singular values $\{\sigma_i\}_{i = 1}^s$,
    and corresponding orthonormal left and right singular vectors
    with coordinates $\alpha^k, \beta^k \in \sphere^{s-1}$.
    Hence,
    we would have 
    \[
        \sigma_2(A - V)
        < \scp{u^k}{(A - V) v^k}
        = \sum_{i = 1}^s \sigma_i \alpha_i^k \beta_i^k
        \quad \forall k \geq k_0
    \]
    almost surely.
    By the convergence of the enire sequence $\scp{u^k}{(A - V) v^k} \to \sigma_1$ 
    for $k \to \infty$ a.s.,
    we necessarily obtain $\alpha_j^k\beta_j^k \to 0$ for $k \to \infty$ 
    and $j \neq 1$ a.s.,
    such that $\alpha_1^k\beta_1^k \to 1$ for $k \to \infty$ a.s.
    This finishes the prove.}
    \qed
\end{enumerate}

\subsection{Proof of Proposition~\ref{prop:conv_rate_angle}:}
\label{proof:conv_rate_angle}
\adj{From Lemma~\ref{lem:asscent_lem_2} 
we have $c_k^2 + \tau_k(a_kb_k + c_kd_k) = \scp{u^{k+1}}{(A - V) v^{k+1}} - \scp{u^k}{(A - V) v^k}$,
and from Proposition~\ref{prop:one_reduced_step_size} it holds $\tau_k(a_kb_k + c_kd_k) > 0$
such that $c_k^2 \leq \scp{u^{k+1}}{(A - V) v^{k+1}} - \scp{u^k}{(A - V) v^k}$.
Summing up the inequality from $k = 0,...,n$ yields 
\begin{align*}
    \sum_{0 \leq k \leq n}
    b_k^2 
    \leq \scp{u^{n+1}}{(A - V) v^{n+1}} - \scp{u^0}{(A - V) v^0}
    \leq \norm{A - V}
\end{align*}
almost surely.
Analogously,
we have $\sum_{0 \leq k \leq n} c_k^2 \leq \norm{A - V}$.
Combining both inequalities 
and taking the expectation yields 
$(n+1) \min_{0 \leq k \leq n} \EE[b_k^2 + c_k^2] \leq \sum_{0 \leq k \leq n} \EE[b_k^2 + c_k^2] \leq 2\norm{A - V}^2$.}
By Lemma \ref{lem:E_c_sqr} we have
\begin{align}
    \EE[b_k^2 + c_k^2\mid u^k, v^k] 
        & = \tfrac1{m-1} \cdot \norm{(I_m - u^k (u^k)^*)(A-V)v^k}^2\nonumber\\
        & \qquad + \tfrac1{d-1} \cdot \norm{(I_m - v^k (v^k)^*)(A-V)^{*}u^k}^2 
\end{align}
\adj{
Due to the law of total expectation,
i.e. $\EE[b_k^2 + c_k^2] = \EE[\EE[b_k^2 + c_k^2 \mid u^k, v^k]]$ for any $k \in \NN$
we obtain
\begin{align*}
    \min_{0 \leq k \leq n}
    &\EE[\norm{(I_m - u^k (u^k)^*)(A-V)v^k}^2 + \norm{(I_m - v^k (v^k)^*)(A-V)^{*}u^k}^2] \\
    &\leq 2\tfrac{\max\{m-1, d-1\}\norm{A - V}^2}{n+1}.
\end{align*}
Let $k^*(n)$ be the sequence realizing the minimum in the latter inequality,
we have by Markov’s inequality for any $\varepsilon> 0$ and $n \in \NN$ that 
\begin{align}
    &\PP(\norm{(I_m - u^{k^*(n)} (u^{k^*(n)})^*)(A-V)v^{k^*(n)}}^2 \notag\\
    &\qquad + \norm{(I_m - v^{k^*(n)} (v^{k^*(n)})^*)(A-V)^{*}u^{k^*(n)}}^2 > \varepsilon) \notag\\
    &\quad\leq \tfrac{\EE[\norm{(I_m - u^{k^*(n)} (u^{k^*(n)})^*)(A-V)v^{k^*(n)}}^2 + \norm{(I_m - v^{k^*(n)} (v^{k^*(n)})^*)(A-V)^{*}u^{k^*(n)}}^2]}{\varepsilon} \notag\\
    &\quad\leq 2\tfrac{\max\{m-1, d-1\}\norm{A - V}^2}{(n+1)\varepsilon},
    \label{eq:PE_project}
\end{align}
and hence 
\begin{align*}
    \left.
    \begin{array}{c}
        \min_{0 \leq k \leq n}
        \PP(\norm{(I_m - u^k (u^k)^*)(A-V) v^k}^2 > \varepsilon) \\
        \min_{0 \leq k \leq n}
        \PP(\norm{(I_m - v^k (v^k)^*)(A-V)^{*} u^k}^2 > \varepsilon)
    \end{array}
    \right\}
    \leq 2\tfrac{\max\{m-1, d-1\}\norm{A - V}^2}{(n+1)\varepsilon}.
\end{align*}
Manipulating each summand in \eqref{eq:PE_project} by $\norm{(A - V) v^{k^*(n)}}$,
respectively $\norm{(A - V)^* u^{k^*(n)}}$
and utilizing Lemma~\ref{lem:norm_proj_relation}
for unormalized projected vector,
i.e. for any $v \in \sphere^{d-1}$ and $z \in \RR^d$ 
holds $\norm{(I_d - vv^*) z} = \norm{z} (1 - \scp{v}{z/\norm{z}}^2)$,
we obtain the inequality in \eqref{eq:PE_project} holds equivalently 
for 
\begin{align*}
    &\norm{(A-V)v^{k^*(n)}} \Bigl(1 - \scp{u^{k^*(n)}}{\tfrac{(A-V)v^{k^*(n)}}{\norm{(A-V)v^{k^*(n)}}}}^2\Bigr) \\
    &\quad + \norm{(A-V)^{*}u^{k^*(n)}} \Bigl(1 - \scp{v^{k^*(n)}}{\tfrac{(A-V)^{*}u^{k^*(n)}}{\norm{(A-V)^{*}u^{k^*(n)}}}}^2\Bigr)
    > \varepsilon.
\end{align*}
Using 
\[
    \norm{A - V} \geq \norm{(A - V) v^k} \geq \scp{u^k}{(A - V) v^k} > \scp{u^0}{(A - V) v^0} = a_0 > \kappa_0
\]
with probability $\delta$ and analogously for $\norm{(A - V)^* u^k}$,
we obtain 
\begin{align*}
    &\PP\Bigl(1 \!-\! \scp{u^{k^*(n)}}{\tfrac{(A-V)v^{k^*(n)}}{\norm{(A-V)v^{k^*(n)}}}}^2
    \!\!+\! 1 \!-\! \scp{v^{k^*(n)}}{\tfrac{(A-V)^{*}u^{k^*(n)}}{\norm{(A-V)^{*}u^{k^*(n)}}}}^2 > \tfrac{\varepsilon}{\kappa_0} 
    \Bigl|\Bigr. \abs{a_0} > \kappa_0\Bigr) \\
    &\quad\leq 2 \tfrac{\max\{m-1, d-1\}\norm{A - V}^2}{(n+1)\varepsilon\delta},
\end{align*}
which yields the assertion.
}
\qed{}

\subsection{Proof of Corollary~\ref{cor:conv_rate_a_k}:}
\label{proof:conv_rate_a_k}

\begin{proof}
    By Assumption~\ref{ass:assump_2}
    we have $\norm{(A - V) v^k} > 0$ for any $k \in \N$ almost surely
    and obtain via factorizing
    \begin{align}
        \label{eq:bound_a_k_1}
        a_k
        = \scp{u^k}{(A - V) v^k}
        = \scp{u^k}{\tfrac{(A - V) v^k}{\norm{(A - V) v^k}}}
        \cdot\norm{(A - V) v^k}
    \end{align}
    and analogously (with $\norm{(A - V)^* u^k} > 0$ for any $k \in \N$) we have 
    \begin{align}
        \label{eq:bound_a_k_2}
        a_k
        = \scp{(A - V)^* u^k}{v^k}
        = \scp{v^k}{\tfrac{(A - V)^* u^k}{\norm{(A - V)^* u^k}}}
        \cdot\norm{(A - V)^* u^k}.
    \end{align}
    Combining \eqref{eq:bound_a_k_1}
    and \eqref{eq:bound_a_k_2} yields
    \begin{align*}
        \label{eq:bounds_gamma}
        a_k^2
        & = 
        \scp{u^k}{\tfrac{(A - V) v^k}{\norm{(A - V) v^k}}}
        \cdot \scp{v^k}{\tfrac{(A - V)^* u^k}{\norm{(A - V)^* u^k}}}
        \cdot \lambda_k \\
        & \geq \min\left(\scp{u^k}{\tfrac{(A - V) v^k}{\norm{(A - V) v^k}}},
        \scp{v^k}{\tfrac{(A - V)^* u^k}{\norm{(A - V)^* u^k}}}\right)^2 \lambda_k,
    \end{align*}
    where $\lambda_k$ is defined in Theorem~\ref{thm:conv_rate}.
    From Cauchy-Schwarz's inequality
    and by construction,
    we have $a_k^2 \leq \lambda_k \leq \norm{A - V}^2$.
    Now,
    we can use Lemma~\ref{lem:norm_proj_relation}
    to rewrite the squared minimum 
    and utilize the same argumentation as in \adj{Theorem}~\ref{thm:conv_rate}
    in \eqref{eq:diff_inner_prod_square}
    and \eqref{eq:conv-rate-differences}
    to obtain 
    \begin{align*}
    0 
    \leq \lambda_k - a_k^2 
    & \leq \lambda_k \left(1 - \min\left(\scp{u^k}{\tfrac{(A - V) v^k}{\norm{(A - V) v^k}}},
    \scp{v^k}{\tfrac{(A - V)^* u^k}{\norm{(A - V)^* u^k}}}\right)^2\right) \\
    & \leq \lambda_k \max\left(\norm[\Big]{u^{k} - \tfrac{(A-V)v^{k}}{\norm{(A-V)v^{k}}}}^{2},
    \norm[\Big]{v^{k} - \tfrac{(A-V)^{*}u^{k}}{\norm{(A-V)^{*}u^{k}}}}^{2}\right) \\
    & \leq \norm{A - V}^2 
    \left[\norm[\Big]{u^{k} - \tfrac{(A-V)v^{k}}{\norm{(A-V)v^{k}}}}^{2} + \norm[\Big]{v^{k} - \tfrac{(A-V)^{*}u^{k}}{\norm{(A-V)^{*}u^{k}}}}^{2}\right].
    \end{align*}
    The assertion follows by using the following inequality 
    and utilizing \eqref{eq:conv-rate-differences} and \eqref{eq:inequality_to_prop}:
    \begin{align*}
    \norm{(A - V)^*(A - V) v^k - a_k^2 v_k}
    & \leq \norm{(A - V)^*(A - V) v^k - \lambda_k v_k}
    + (\lambda_k - a_k^2).
    \qedhere
    \end{align*}
\end{proof}

\printbibliography

\end{document}